\documentclass[10pt]{article}
\usepackage{xcolor}
\usepackage{color,soul}
\usepackage{soul}

\usepackage{physics}
\usepackage{amsmath}
\usepackage{amssymb}
\usepackage{amsthm}
\usepackage{fullpage}
\usepackage{graphics}
\usepackage{graphicx}
\usepackage{subcaption}                             
\usepackage{caption}
\usepackage[font=footnotesize,labelfont=bf]{caption}
\usepackage[draft]{todonotes} 
\usepackage{mathtools}
\usepackage{bbm}
\usepackage{multicol}
\usepackage[shortlabels]{enumitem}
\newlist{abbrv}{itemize}{1}
\setlist[abbrv,1]{label=,labelwidth=1in,align=parleft,itemsep=0.1\baselineskip,leftmargin=!}
\newcommand{\mr}{\mathbf}

\newtheorem{thm}{Theorem}[section]

\newtheorem{lm}{Lemma}[section]
\newtheorem{prop}{Proposition}[section]
\newtheorem{rmk}{Remark}[section]
\newtheorem{cor}{Corollary}[section]
\newtheorem{claim}{Claim}[section]

\newcommand{\prob}{\mathbb{P}}
\newcommand{\intZ}{\mathbb{Z}}
\newcommand{\realR}{\mathbb{R}}
\newcommand{\complexC}{\mathbb{C}}
\newcommand{\FGUE}{F_{GUE}}

\newcommand{\iid}{i.i.d.\,}
\newcommand{\polylog}{\mathrm{Li}}

\newcommand{\dist}{{\rm{dist}}\,}
\newcommand{\EE}{\mathbb{E}}
\newcommand{\Var}{\mathrm{Var}}

\newcommand{\FS}{F_2}

\newcommand{\FB}{F_B}
\newcommand{\FU}{F_U}

\newcommand{\Ks}{\mathcal{K}^{(2)}}

\newcommand{\inodes}{S}
\newcommand{\conf}{\mathcal{X}}

\renewcommand{\dd}{{\mathrm d}}
\newcommand{\ii}{\mathrm{i}}
\newcommand{\ddbar}[1]{\frac{{\mathrm d}#1}{2\pi {\mathrm i}#1}}
\newcommand{\ddbarr}[1]{\frac{{\mathrm d}#1}{2\pi {\mathrm i}}}

\newcommand{\LL}{{\rm left}}
\newcommand{\RR}{{\rm right}}
\newcommand{\zz}{\mathbf{z}}
\newcommand{\rr}{\mathbbm{r}}
\newcommand{\yy}{\ell}

\newcommand{\roots}{R}

\newcommand{\consts}{\mathcal{C}_N^{(2)}}

\newcommand{\hftn}{\mathfrak{h}}
\newcommand{\gee}{G}

\newcommand{\Cr}[1]{#1}
\newcommand{\Cb}[1]{#1}

\numberwithin{equation}{section} 
\numberwithin{thm}{section}

\author{Zhipeng Liu\footnote{Courant Institute of Mathematical Sciences, New York University, New York, NY 10012
\newline email: \texttt{zhipeng@cims.nyu.edu}}}
\date{\today}

\begin{document}
\title{Height fluctuations of stationary TASEP on a ring in relaxation time scale}
\maketitle

\begin{abstract}
We consider the totally asymmetric simple exclusion process on a ring with stationary initial conditions. The crossover between KPZ dynamics and equilibrium dynamics occurs when time is proportional to the $3/2$ power of the ring size. We obtain the limit of \Cr{the} height function along the direction of \Cr{the} characteristic line in this time scale. The two-point covariance function in this scale is also discussed.
\end{abstract}

\section{Introduction}


In this paper we consider the totally asymmetric simple exclusion process (TASEP) on a ring of size $L$ which we denote by $\intZ_L$. The dynamics of TASEP on the ring is the same as that of TASEP on $\intZ$ except the particle at the site $\overline{L-1}$, once it jumps, moves to the site $\overline0$ if $\overline0$ is empty, here the $\overline{i}$ denotes the element $i\pmod L$ in $\intZ_L$ for $i\in\{0,1,\cdots, L-1\}$. Let $\eta_i=\eta_i(t)$ the occupation variable of this model, $0\le i\le L-1$. $\eta_i$ is $1$ if the site $\overline{i}$ is occupied or $0$ if the site $\overline{i}$ is empty. We extend the occupation variable to $\intZ$ periodically by defining $\eta_i(t)=\eta_{i+L}(t)$ for all $i\in\intZ$. Define the following height function
\begin{equation}
h_t(\yy)=\begin{cases}
2J_0(t)+\sum_{j=1}^{\yy}(1-2\eta_j(t)),&\yy\ge 1,\\
2J_0(t),&\yy=0,\\
2J_0(t)-\sum_{j=\yy+1}^0(1-2\eta_j(t)),&\yy\le -1,
\end{cases}
\end{equation}
where $J_0(t)$ counts the number of particles jumping through the bond from $0$ to $1$ on $\intZ_L$ during the time interval $[0,t]$. Note that $h_t(\yy)-h_0(\yy)=2J_\yy(t)$, where $J_\yy(t)$ counts the number of particles jumping through the bond from $\yy\pmod L$ to $\yy+1\pmod L$ on $\intZ_L$ during the time interval $[0,t]$. Although $\eta_\yy(t), J_\yy(t)$ are both periodic in $\yy$, $h_t(\yy)$ is not periodic except when the system is half-filled. Indeed, we have $h_t(\yy+L)=h_t(\yy)+(L-2N)$ for all $\yy\in\intZ$ and  $t\ge 0$, where $N=\sum_{j=0}^{L-1}\eta_j$ is the number of particles on the ring.

We are interested in the fluctuations of $h_t(\yy)$ when $t$ and $\yy$ both increase \Cr{with} order $O(L^{3/2})$,  and $L, N$ go to infinity proportionally. The scale $t=O(L^{3/2})$ is called the relaxation time scale, which was first studied by Gwa and Spohn \cite{Gwa-Spohn92}. At this relaxation time scale, one expects to see a crossover between the KPZ dynamics and the Gaussian dynamics and hence the fluctuations are of great interest to both math and physics communities. The crossover limiting distributions were obtained only recently by Prolhac \cite{Prolhac16} and Baik and Liu \cite{Baik-Liu16}. In \cite{Prolhac16}, Prolhac obtained (not rigorously) the limit of the current fluctuations for step, flat and stationary initial conditions in the half particle system (with the restriction $L=2N$). Independently, Baik and Liu also obtained the limit in a more general setting of $N$ and $L$ for flat and step initial conditions in \cite{Baik-Liu16}\footnote{The formulas of the limiting distribution in two papers \cite{Prolhac16} and \cite{Baik-Liu16} are slightly different and it is yet to be proved that they are indeed the same. The numeric plots show  that they do agree.}. 
The main goal of this paper is to extend the work of \cite{Baik-Liu16} to the stationary initial condition \Cr{case} and prove the rigorous limit theorem of $h_t(\yy)$ in \Cb{the} relaxation time scale. 
Compared to \cite{Prolhac16}, there are some other differences \Cr{besides} the rigorousness: We consider a more general setting of stationary initial conditions than \Cb{the half-filled one in} \cite{Prolhac16}, and \Cb{a more general object}, the height function $h_t(\yy)$, than the current in \cite{Prolhac16}, which is equivalent to $h_t(0)$. Hence the limiting distribution obtained in this paper, $\FU(x;\tau,\gamma)$ in Theorem~\ref{thm:limit_height_uniform}, contains two parameters of time $\tau$ and location $\gamma$, in contrast to that of only time parameter in \cite{Prolhac16}.

Due to the ring structure, the number of particles is invariant. Hence it is natural to consider the following \emph{uniform} initial condition of $N$ particles: initially all possible configurations of $N$ particles on the ring of size $L$ are of equal probability, i.e., ${L\choose N}^{-1}=\frac{N!(L-N)!}{L!}$. This initial condition is stationary, and is the unique one for fixed number of particles $N$ and ring size $L$ \cite{Meakin-Ramanlal-Sander-Ball86}. 

For this uniform initial condition, there is a \emph{characteristic} line $\yy=(1-2\rho)t$ in the space-time plane\footnote{\Cr{It is the characteristic line  of the related Burger's equation in the space-time plane. See the appendix of }\cite{Baik-Liu16b} \Cr{for discussions on the Burger's equation related to TASEP on a ring}.}, here $\rho=NL^{-1}$ is the density of the system. The main theorem of this paper is \Cb{about} the fluctuations of \Cr{$h_{t}(\yy)$} near the characteristic line in \Cb{the} relaxation time scale. 

\begin{thm}
	\label{thm:limit_height_uniform}
	Let $c_1$ and $c_2$ be two fixed constants satisfying $0<c_1<c_2<1$. Suppose $N_L$ is a sequence of integers such that $c_1L\le N_L\le c_2L$ for all \Cr{sufficiently large} $L$.  We consider the TASEP on a ring of size $L$ with  $N_L$ particles. Assume that they satisfy the uniform initial condition. Denote $\rho_L=N_L/L$. 
	Let $\tau$ and $w$ be two fixed constants satisfying $\tau>0$ and $w\in\realR$.  Suppose 
	\begin{equation}
	\label{eq:parameter_time}
	t_L=\frac{\tau}{\sqrt{\rho_L(1-\rho_L)}} L^{3/2}.
	\end{equation}
	Then along the line
	\begin{equation}
	\label{eq:aux_2016_09_03_01}
	\yy_L=(1-2\rho_L)t_L+2w(\rho_L(1-\rho_L))^{1/3}t_L^{2/3},
	\end{equation}
	 we have
	\begin{equation}
	\label{eq:aux_2016_09_03_03}
	\lim_{L\to\infty}\prob\left(\frac{h_{t_L}(\yy_L)-(1-2\rho_L)\yy_L-2\rho_L(1-\rho_L)t_L}{-2\rho_L^{2/3}(1-\rho_L)^{2/3}t_L^{1/3}}\le x\right)=\FU(\tau^{1/3}x;\tau,2w\tau^{2/3})
	\end{equation}
	for each $x\in\realR$. Here $\FU(x;\tau,\gamma)$ is a distribution function defined in~\eqref{eq:def_FB} for any $\tau>0$ and $\gamma=2w\tau^{2/3}\in\realR$. It satisfies $\FU(x;\tau,\gamma)=\FU(x;\tau,\gamma+1)$ and $\FU(x;\tau,\gamma)=\FU(x;\tau,-\gamma)$.
\end{thm}
\begin{rmk}
	\label{rmk:01}
	In \cite{Prolhac16}, Prolhac obtained~\eqref{eq:aux_2016_09_03_03} when \Cr{$\yy_L=0$} and $\rho_L=1/2$ (and hence $w=0$, $\gamma=0$) with a different formula of the limiting distribution. His proof, as mentioned before, is not completely rigorous.
\end{rmk}

Note that if we write $\gamma=2w\tau^{2/3}$, then the line~\eqref{eq:aux_2016_09_03_01} can be rewritten as
\begin{equation}
\label{eq:aux_2016_10_06_01}
\yy_L=(1-2\rho_L)t_L+\gamma L.
\end{equation}
This expression gives an intuitive reason why the limiting function $\FU(x;\tau,\gamma)$ is periodic on $\gamma$: It is the periodicity of the shifted height function $h_{t_L}(\yy_L+L)-(1-2\rho_L)(\yy_L+L)=h_{t_L}(\yy_L)-(1-2\rho_L)\yy_L$.


To better understand the parametrization in the above theorem, we compare it with the infinite TASEP with stationary condition, i.e., the stationary TASEP on $\intZ$. Suppose initially each site in $\intZ$ is occupied independently with probability $p$. Then the height fluctuation converges along the line 
$\yy=(1-2p)t+ 2w(p(1-p))^{1/3}t^{2/3}$ for any given constant $w\in\realR$, see \cite{Ferrari-Spohn06, Baik-Ferrari-Peche10},
\begin{equation}
\label{eq:aux_2016_09_03_04}
\lim_{t\to\infty}\prob\left(\frac{h_t(\yy)-(1-2p)\yy-2p(1-p)t}{-2p^{2/3}(1-p)^{2/3}t^{1/3}}\le x\right)=F_w(x),\quad x\in\realR,
\end{equation}
where $F_w(s)$ is the Baik-Rains distribution defined in \cite{Baik-Rains00}\footnote{In \cite{Baik-Rains00}, $F_w(s)$ was denoted by $H(s+w^2;w/2,-w/2)$.}.
Theorem~\ref{thm:limit_height_uniform} of this paper shows that for the stationary TASEP on a ring with uniform initial condition in relaxation time scale, similar limiting laws hold near the characteristic line. 
The difference is that 
for the ring TASEP, the fluctuations have a periodicity on the parameter $\gamma=2w\tau^{2/3}$, which is not present in the infinite TASEP model.

\vspace{0.3cm}

The leading terms $(1-2\rho_L)\yy_L$ and $2\rho_L(1-\rho_L)t_L$ in $h_{t_L}(\yy_L)$ can be explained as \Cb{follows}. The first term $(1-2\rho_L)\yy_L$ measures the change of height along the direction $\yy_L$: For fixed $t_L$, $h_{t_L}(\yy_L)-h_{t_L}(0)$ grows as $(1-2\rho_L)\yy_L$ in the leading order \Cb{since} $h_{t_L}(\yy_L+L)=h_{t_L}(\yy_L)+(1-2\rho_L)L$. The second term $2\rho_L(1-\rho_L)t_L$ measures the time-integrated current at a fixed location: $h_{t_L}(\yy_L)-h_0(\yy_L)=2J_0(t_L)$ which grows as $2\rho_L(1-\rho_L)t_L$ in the leading order.

\vspace{0.3cm}
As an application of Theorem~\ref{thm:limit_height_uniform}, we can express the limit of two-point covariance function in terms of $\FU(x;\tau,\gamma)$. Recall the occupation variable $\eta_\yy(t)$ at the beginning of the paper. Define the two-point covariance function
\begin{equation}
S(\yy;t):=\EE\left(\eta_\yy(t)\eta_0(0)\right)-\rho^2
\end{equation}
where $\rho=N/L$ is the system density. It is known that for the stationary TASEP, there is a relation between this two-point function $S(\yy;t)$ and the height function $h_\yy(t)$: $8S(\yy;t)=\Var(h_t(\yy+1))-2\Var(h_t(\yy))+\Var(h_t(\yy-1))$. This relation was proved for the infinite TASEP in \cite{Prahofer-Spohn02a} but the proof is also valid for TASEP on a ring after minor modifications. Using this identity and the tail estimate which is provided in the appendix~\ref{sec:estimates}, we obtain the following result. The proof is almost the same as that for the stationary TASEP on $\intZ$, see \cite{Baik-Ferrari-Peche14}, and hence we omit it.


\begin{cor}
	\label{cor:space_time_covariance_uniform}
	Suppose $N_L, t_L$ and $\yy_L$ are defined as in Theorem~\ref{thm:limit_height_uniform} with the same constants $\tau>0$ and $\gamma=2w\tau^{2/3}\in\realR$. Then we have
	\begin{equation}
	\lim_{L\to\infty}\frac{2t_L^{2/3}{S(\yy_L;t_L)}}{\rho_L^{2/3}(1-\rho_L)^{2/3}}=g''_U(\gamma;\tau),
	\end{equation}
	if integrated over smooth functions in $\gamma$ with compact support, where
	\begin{equation}
	g_U(\gamma;\tau):=\tau^{2/3}\int_{\realR}x^2\dd \FU(x;\tau,\gamma).
	\end{equation}
\end{cor}


\vspace{0.3cm}
Another application is that one can obtain the height fluctuations for other stationary TASEP on a ring. Note that the uniform initial conditions with $N=0,1,\cdots, L$ form a complete basis for all stationary initial conditions. Hence we may apply Theorem~\ref{thm:limit_height_uniform} for other stationary initial conditions. One  example is the Bernoulli condition. Suppose initially each site of the ring is occupied independently with probability $p$, where $p$ is a constant satisfies $0<p<1$. 
Then we have the following result.
\begin{cor}
	\label{cor:height_fluctuations_Bernoulli}
	Suppose $p\in(0,1)$ is a fixed constant. We consider the TASEP on the ring of size $L$ with Bernoulli initial condition of parameter $p$. Suppose $w\in\realR$, $\tau>0$ and $x\in\realR$ are fixed constants. Denote
	\begin{equation}
	\begin{split}
	t_L		&=\frac{\tau}{\sqrt{p(1-p)}} L^{3/2},\\
	\yy_L 	&=(1-2p)t_L+2w(p(1-p))^{1/3} t_L^{2/3},
	\end{split}
	\end{equation}
	Then
	\begin{equation}
	\lim_{L\to\infty} \prob \left(\frac{h_{t_L}(\yy_L)-(1-2p)\yy_L-2p(1-p)t_L}{-2p^{2/3}(1-p)^{2/3}t_L^{1/3}}\le x\right) = \FB(\tau^{1/3}x;\tau,2w\tau^{2/3}).
	\end{equation}
	Here $\FB(x;\tau,\gamma)$ is a distribution function for arbitrary $\tau>0$ and $\gamma=2w\tau^{2/3}\in\realR$, given by
	\begin{equation}
	\label{eq:aux_2016_09_01_02}
	\FB(x;\tau,\gamma):=\frac{1}{2\sqrt{2\pi}\tau}\int_{\realR} e^{-\frac{(y-\gamma)^2}{8\tau^2}}\FU\left(x+\frac{\gamma^2-y^2}{4\tau};\tau,y\right)\dd y.
	\end{equation}
\end{cor}

A formal proof is as \Cr{follows}. Assume there are $pL+y\sqrt{p(1-p)}L^{1/2}$ particles initially. By applying Theorem~\ref{thm:limit_height_uniform}, we obtain that
\begin{equation}
\prob\left(\left.\frac{h_{t_L}(\yy_L)-(1-2p)\yy_L-2p(1-p)t_L}{-2p^{2/3}(1-p)^{2/3}t_L^{1/3}}\le x\ \right|\  pL+y\sqrt{p(1-p)}L^{1/2} \mbox{ particles with uniform initial condition }\right)
\end{equation}
converges to
\begin{equation}
\FU\left(\tau^{1/3}x-\gamma y-\tau y^2;\tau,\gamma+2y\tau\right)
\end{equation}
as $L\to\infty$, where $\gamma=2w\tau^{2/3}$. Together with the central limit theorem, we obtain
\begin{equation}
\lim_{L\to\infty}\left(\frac{h_{t_L}(\yy_L)-(1-2p)\yy_L-2p(1-p)t_L}{-2p^{2/3}(1-p)^{2/3}t_L^{1/3}}\le x\right)=\frac{1}{\sqrt{2\pi}}\int_{\realR}e^{-y^2/2}\FU\left(\tau^{1/3}x-\gamma y-\tau y^2;\tau,\gamma+2y\tau\right)\dd y.
\end{equation}
By a simple change of variables we arrive at~\eqref{eq:aux_2016_09_01_02}. This argument can be made rigorous by a simple tail estimate on the number of particles and then by the dominated convergence theorem. Since the argument is standard, we omit the details.

\vspace{0.3cm}

Recall that $\FU(x;\tau,\gamma)$ is symmetric on $\gamma$. Hence by using the formula~\eqref{eq:aux_2016_09_01_02} we have $\FB(x;\tau,\gamma)=\FB(x;\tau,-\gamma)$. However, different from $\FU(x;\tau,\gamma)$, we do not expect $\FB(x;\tau,\gamma)=\FB(x;\tau,\gamma+1)$. It is because by definition $h_t(\yy+L)-h_t(\yy)-(1-2p)L=-2\sum_{j=\yy+1}^{\yy+L}(\eta_j(t)-p)\approx -2L^{1/2}\sqrt{p(1-p)}\chi$ where $\chi$ is a standard Gaussian random variable. Hence formally
\begin{equation}
\label{eq:aux_2016_09_18_01}
\frac{h_{t_L}(\yy_L+L)-(1-2p)(\yy_L+L)-2p(1-p)t_L}{-2p^{2/3}(1-p)^{2/3}t_L^{1/3}}\approx\frac{h_{t_L}(\yy_L)-(1-2p)\yy_L-2p(1-p)t_L}{-2p^{2/3}(1-p)^{2/3}t_L^{1/3}}+\frac\chi{\tau^{1/3}}.
\end{equation}
Here the two random variables on the right hand side of~\eqref{eq:aux_2016_09_18_01} are not necessarily independent. This relation still strongly indicates that $\FB(\tau^{1/3}x;\tau,\gamma+1)$ is not the same as $\FB(\tau^{1/3}x;\tau,\gamma)$.

\vspace{0.3cm}
The organization of this paper is as \Cb{follows}. In Section~\ref{sec:limiting_distribution} we give the explicit formula and some properties of $\FU(x;\tau,\gamma)$. The proof of Theorem~\ref{thm:limit_height_uniform} is given in Section~\ref{sec:formulas} and~\ref{sec:asymptotics}: The finite time distribution formula is provided in Section~\ref{sec:formulas} and then the asymptotics in Section~\ref{sec:asymptotics}. Finally in the appendix \ref{sec:estimates} we give some tail bounds related to the distribution function $\FU(x;\tau,\gamma)$.

\section*{Acknowledgement}
The author would like to thank Jinho Baik, Ivan Corwin and Peter Nejjar for useful discussions.






\section{Limiting distribution $\FU$}
\label{sec:limiting_distribution}

The limiting distribution $\FU(x;\tau,\gamma)$ is defined as following
\begin{equation}
\label{eq:def_FB}
\FU(x;\tau,\gamma) = -\oint\frac{\dd}{\dd x}\left(e^{xA_1(z)+\tau A_2(z) +2B(z)}\det\left(I -\Ks_{z;x}\right)\right)\frac{\dd z}{\sqrt{2\pi} \ii z^2}
\end{equation}
where the integral is \Cb{along} an arbitrary simple closed contour within the disk $|z|<1$ and with $0$ inside. The terms $A_i(z)$ are given by
\begin{equation}
\label{eq:def_constants_A}
A_1(z) = -\frac{1}{\sqrt{2\pi}} \polylog_{3/2}(z),
\qquad
A_2(z) = -\frac{1}{\sqrt{2\pi}} \polylog_{5/2}(z),
\end{equation}
and $B(z)$ is given by
\begin{equation}
\label{eq:def_constant_B}
B(z) = \frac{1}{4\pi} \int_0^z \frac{(\polylog_{1/2}(y))^2}{y} \dd y.
\end{equation}
Here $\polylog_s(z)$ is the polylogarithm function \Cr{defined as follows: When $|z|<1$, $\polylog_s(z):=\sum_{k=1}^\infty\frac{z^k}{k^s}$, and it has an analytic continuation}
\begin{equation}
\polylog_{s}(z)=\frac{z}{\Gamma(s)}\int_0^\infty\frac{x^{s-1}}{e^{x}-z}\dd x
\end{equation}
\Cr{for all $z\in\complexC\setminus\realR_{\ge 1}$.}

The operator $\Ks_{z;x}$ is defined on the set $\inodes_{z,\LL} = \{\xi: e^{-\xi^2/2}=z, \Re(\xi)<0\}$ with kernel
\begin{equation}
\label{eq:def_inf_kernel_step}
\Ks_{z;x}(\xi_1,\xi_2) = \Ks_{z;x}(\xi_1,\xi_2;\tau,\gamma)
= \sum_{\eta \in \inodes_{z,\LL}}
\frac{e^{\Phi_z(\xi_1;x,\tau)  +\Phi_z(\eta;x,\tau)		 
		+\frac{\gamma}{2}(\xi_1^2-\eta^2)}}
{\xi_1\eta(\xi_1+\eta)(\eta+\xi_2)},
\end{equation}
where
\begin{equation}
\label{eq:def_Phi}
\Phi_z(\xi;x,\tau) = -\frac13\tau\xi^3	 +x\xi 	
-\sqrt{\frac{2}{\pi}}\int_{-\infty}^{\xi} \polylog_{1/2}(e^{-\omega^2/2})\dd \omega, \qquad \xi\in \inodes_{z,\LL}.
\end{equation}

\vspace{0.3cm}

The terms $A_i(z), B(z)$ and $\Ks_{z;x}$ are defined in \cite{Baik-Liu16}. They appeared in the\Cb{ two-parameter family of} limiting distributions \Cb{$\FS(x;\tau,\gamma)$} of TASEP on a ring with step initial condition in \Cb{the} relaxation time scale\Cr{. More explicitly, $\FS(x;\tau,\gamma)$ has an integral formula which is similar to}~\eqref{eq:def_FB}
\begin{equation}
\label{eq:aux_2016_09_03_02}
\FS(x;\tau,\gamma)=\oint e^{xA_1(z)+\tau A_2(z) +2B(z)}\det\left(I -\Ks_{z;x}\right)\ddbar{z},
\end{equation}
see (4.10) of \cite{Baik-Liu16}.
 It is known that \Cb{the terms $A_i(z), B(z)$ and $\Ks_{z;x}$} are well defined and bounded uniformly on the choice of $z$ (but the bound may depend on the contour). 
Furthermore, the Fredholm determinant $\det\left(I -\Ks_{z;x}\right)$ is periodic and symmetric on $\gamma$, which implies $\FS(x;\tau,\gamma)=\FS(x;\tau,\gamma+1)$ and $\FS(x;\tau,\gamma)=\FS(x;\tau,-\gamma)$.

To ensure $\FU(x;\tau,\gamma)$ in~\eqref{eq:def_FB} is well defined, we still need to check that the derivative in the integrand exists and is uniformly bounded. The only non-trivial part is to check $\frac{\dd}{\dd x}\det\left(I-\Ks_{z;x}\right)$. This can be proved by directly using the super-exponential decaying property of the kernel. The argument is standard and we do not provide details. Alternately, our analysis in Section~\ref{sec:asymptotics_Delta_determinant} also implies \Cb{that}  $\frac{\dd}{\dd x}\det\left(I-\Ks_{z;x}\right)$ is a limit of a uniformly bounded sequence hence it is also uniformly bounded. See Lemma~\ref{lm:moments_convergent} and~\ref{lm:uniform_bound_difference_determinant}.



\vspace{0.3cm}

As we discussed in Remark \ref{rmk:01}, the limiting distribution when $\gamma=0$ was obtained in \cite{Prolhac16}. Numeric plots of our formula $\FU(x;\tau,0)$ match the limiting distribution obtained in \cite{Prolhac16} well, see Figure~\ref{fig:FU_BR} in this paper and Fig.5.b in \cite{Prolhac16}. However, a rigorous proof of the equivalence on $\FU(x;\tau,0)$ and their formula (see (10) of \cite{Prolhac16}) is still missing.

\begin{figure}
	\centering
	\begin{minipage}{.42\textwidth}
		\includegraphics[scale=0.4]{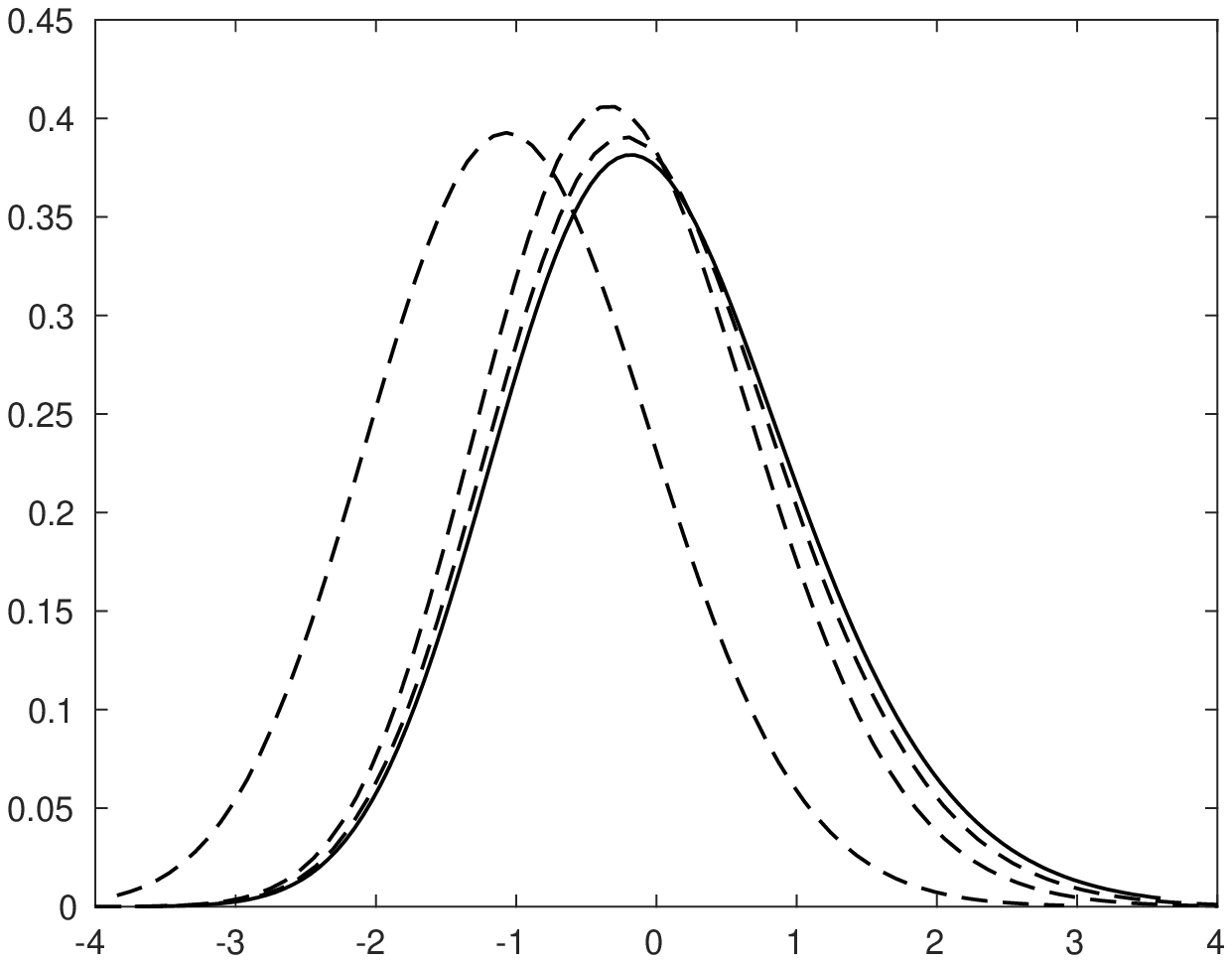}
		\caption{The three dashed lines are, from left to right, density functions of $\FU(\tau^{1/3}x;\tau,0)$ with $\tau=1$, $0.1$, and $0.02$ respectively. And the solid line is the density function of Baik-Rains distribution $F_0(x)$.}
		\label{fig:FU_BR}
	\end{minipage}
	\qquad
	\begin{minipage}{.42\textwidth}
		\includegraphics[scale=0.3]{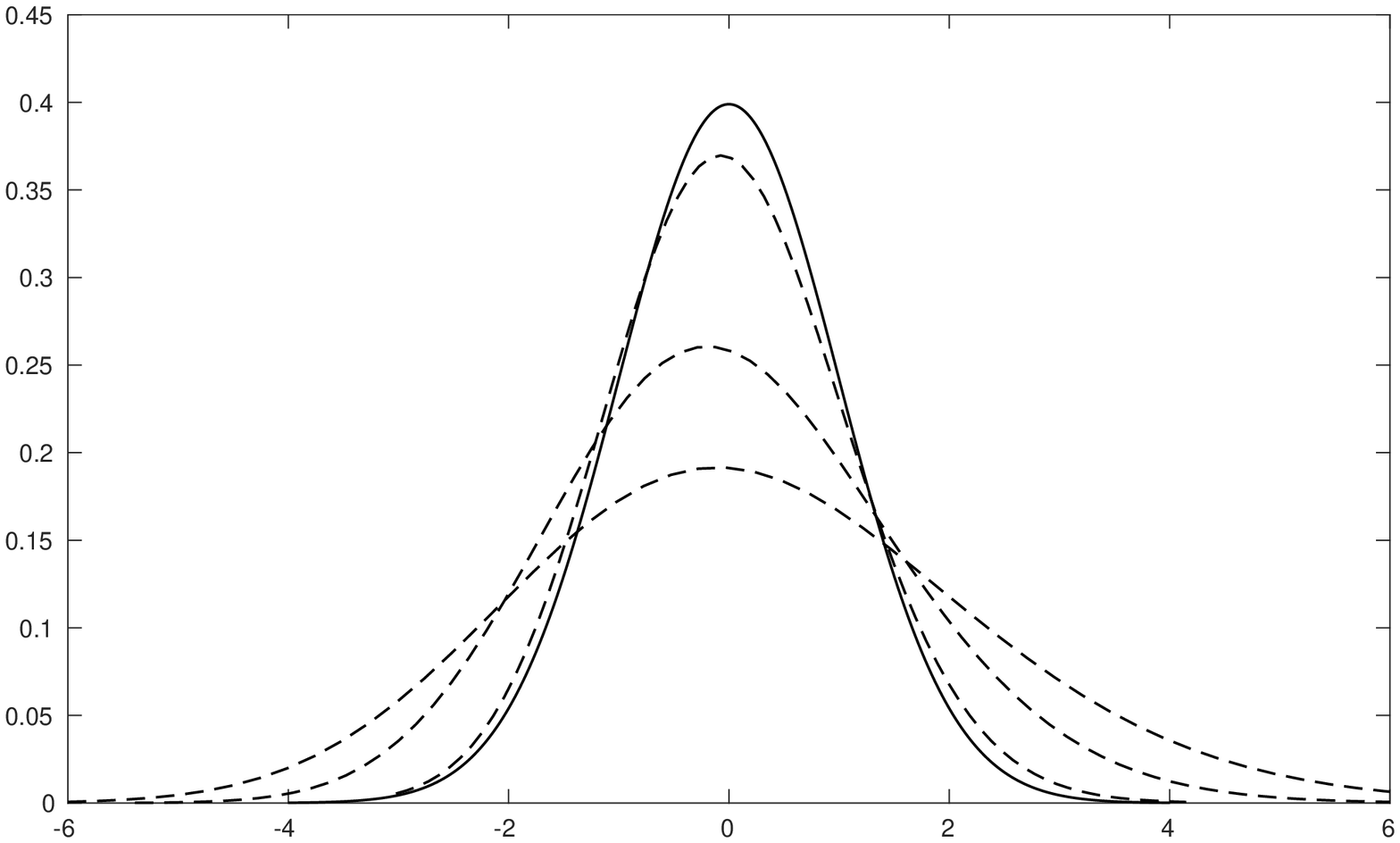}
		\caption{The three dashed lines are, from bottom to top\Cr{ (along $x=0$)}, density functions of $\FU\left(-\tau+\frac{\pi^{1/4}}{\sqrt{2}}x\tau^{1/2};\tau,0\right)$ with $\tau=0.02$, $0.1$, and $1$ respectively. And the solid line is the density function of the standard Gaussian distribution.}
		\label{fig:FU_Gaussian}
	\end{minipage}
\end{figure}

\vspace{0.3cm}
For any fixed $\tau>0$ and $\gamma\in\realR$, the function $\FU(x;\tau,\gamma)$ is a distribution function. The proof is not trivial and we \Cb{provide} it in the appendix~\ref{sec:estimates}. Similar to $\FS(x;\tau,\gamma)$, the function $\FU(x;\tau,\gamma)$ has the properties $\FU(x;\tau,\gamma+1)=\FU(x;\tau,\gamma)$ and $\FU(x;\tau,\gamma)=\FU(x;\tau,-\gamma)$. By using the following simple identity (see (11) of \cite{Derrida-Lebowitz98})
\begin{equation} 
\EE(h_{t_L}(\yy_L))=(1-2\rho_L)\yy_L+2\rho_L(1-\rho_L)t_L+\frac{2\rho_L(1-\rho_L)t_L}{L-1}
\end{equation}
which can also be checked directly from the definition, we have
\begin{equation}
\int_{\realR} x\dd \FU(x;\tau,\gamma)=-\tau.
\end{equation}
The rigorous proof of this identity is similar to Corollary~\ref{cor:space_time_covariance_uniform}. \Cb{Thus} we do not provide details here.
\vspace{0.2cm}

Besides, we expect the following small $\tau$ and large $\tau$ limits of $\FU(x;\tau,\gamma)$:
	\begin{enumerate}[(1)]
		\item For any fixed $x,w\in\realR$, we have (see Figure~\ref{fig:FU_BR} for an illustration)
		\begin{equation}
		\label{eq:aux_2016_09_03_05}
		\lim_{\tau\to0} \FU(\tau^{1/3}x;\tau,2w\tau^{2/3})=F_w(x).
		\end{equation}

		\item For any fixed $\gamma,x\in\realR$, we have (see Figure~\ref{fig:FU_Gaussian} for an illustration)
		\begin{equation}
		\lim_{\tau\to\infty}\FU\left(-\tau+\frac{\pi^{1/4}}{\sqrt{2}}x\tau^{1/2};\tau,\gamma\right)=\frac{1}{\sqrt{2\pi}}\int_{-\infty}^{x}e^{-y^2/2}\dd y.
		\end{equation}
	\end{enumerate}

\section{An exact formula of height distribution}
\label{sec:formulas}
In this section, we prove  an exact formula for the height function with uniform initial condition. This formula turns out to be suitable for later asymptotic analysis.

Before stating the results, we need to introduce some notations. Most of these notations are the same as in~\cite{Baik-Liu16}. Hence we just go through \Cb{them} quickly without further discussions. See Section 7 of \cite{Baik-Liu16} for more details.

We fix $L$ and $N$ in this section, and denote
\begin{equation}
\rho =\frac{N}{L}
\end{equation}
the density of the system.

For \Cb{each} $\zz\in\complexC$, define a polynomial
\begin{equation}
\label{eq:aux_2016_09_05_03}
q_{\zz}(w) = w^N(w+1)^{L-N} -\zz^L
\end{equation}
and its root set
\begin{equation}
\label{eq:aux_2016_10_05_01}
\roots_{\zz} =\{w: q_\zz(w)=0\}.
\end{equation}
When $\zz=0$, $\roots_{\zz}$ is a degenerated set of two points $0$ and $1$ with multiplicities $N$ and $L-N$ respectively. On the other hand, when $\zz\to\infty$, $\roots_{\zz}$ is \Cb{asymptotically} equal to a set of $L$  equidistant points on a circle $|w|=|\zz|$. For our purpose, we focus on the case when
\begin{equation}
\label{eq:r_bounds}
0<|\zz|<\rr_0:=\rho^\rho(1-\rho)^{1-\rho}.
\end{equation}
For such $\zz$, $\roots_\zz$ contains  $L-N$ points in the half plane $\{w: \Re(w) < -\rho\}$ and $N$ points in the second half plane $\{w:\Re(w)> -\rho\}$. We denote $\roots_{\zz,\LL}$ and $\roots_{\zz,\RR}$ the sets of these $L-N$ and $N$ points respectively. Then we define
\begin{equation}
q_{\zz,\LL}(w) = \prod_{u\in\roots_{\zz,\LL}}(w-u), \qquad q_{\zz,\RR}(w) = \prod_{v\in\roots_{\zz,\RR}} (w-v),
\end{equation}
which are two monic polynomials with root sets $\roots_{\zz,\LL}$ and $\roots_{\zz,\RR}$ respectively. These two functions satisfy the following equation
\begin{equation}
q_{\zz,\LL}(w) q_{\zz,\RR}(w) =q_\zz(w)
\end{equation}
for all $w\in\complexC$.

For $\zz\in\complexC$ satisfying~\eqref{eq:r_bounds} and arbitrary $k,\yy\in\intZ$, we define a kernel $K_{\zz;k,\yy}^{(2)}$ acting on $\ell^2(\roots_{\zz,\LL})$ as \Cb{follows}
\begin{equation}
\label{eq:def_kernel}
K_{\zz;k,\yy}^{(2)}(u,u') =f_2(u)\sum_{v\in\roots_{\zz,\RR}}\frac{1}{(u-v)(u'-v)f_2(v)},\qquad u,u'\in\roots_{\zz,\LL},
\end{equation}
where the function $f_2:\roots_\zz\to \complexC$ is defined by 
\begin{equation}
\label{eq:def_f}
f_2(w)= f_2(w;k,\yy):=\begin{dcases}
\frac{\left(q_{\zz,\RR}(w)\right)^2 w^{-2N-k+2} (w+1)^{-\yy+k+1}e^{tw}}{w+\rho},
&\quad w\in\roots_{\zz,\LL},\\
\frac{\left(q'_{\zz,\RR}(w)\right)^2 w^{-2N-k+2} (w+1)^{-\yy+k+1}e^{tw}}{w+\rho},
&\quad w\in\roots_{\zz,\RR}.
\end{dcases}
\end{equation}

We also define a function
\begin{equation}
\label{eq:def_consts}
\consts(\zz;k,\yy) = \frac{\prod_{u\in\roots_{\zz,\LL}} (-u)^{k+N-1}
	\prod_{v\in\roots_{\zz,\RR}} {(v+1)^{-\yy+L-N+k}} e^{tv}}
{\prod_{u\in\roots_{\zz,\LL}}\prod_{v\in\roots_{\zz,\RR}}(v-u)}.
\end{equation}
$K_{\zz;k,\yy}^{(2)}$ and $\consts(\zz;k,\yy)$ are the same as $K_\zz^{(2)}$ and $\consts(\zz)$ in \cite{Baik-Liu16} (with $\yy$ and $k$ replaced by $a$ and $k-N$ respectively) but we emphasize the parameters $k$ and $\yy$ for our purpose. 

Finally we denote $\Delta_k$ the difference operator
\begin{equation}
\label{eq:def_delta}
\Delta_k f(k)=f(k+1)-f(k)
\end{equation}
for arbitrary function $f:\intZ\to \complexC$. For an example, $\Delta_k\consts(\zz;k,\yy)=\consts(\zz;k+1,\yy)-\consts(\zz;k,\yy)$. 

\vspace{0.5cm}

Now we state the formula for the distribution function of $h_t(\yy)$.

\begin{thm}
\label{thm:current_fluctuation}
Suppose $\yy$ and $b$ are both integers satisfying $b\equiv \yy\pmod{2}$. For the $N$-particle TASEP on the ring of size $L$ with uniform initial condition, the distribution of the height function is given by
\begin{equation}
\label{eq:height_fluctuation}
\prob\left(h_t(\yy)\ge b\right) 
   = \frac{(-1)^{N+1}}{{L\choose N}}\oint \Delta_k\left( \consts(\zz; k,\yy+1)\cdot \det\left(I + K_{\zz;k,\yy+1}^{(2)}\right)\right) \frac{\dd \zz}{2\pi \ii \zz^{L+1}},
\end{equation}
where 
\begin{equation}
\label{eq:aux_2016_09_04_01}
k = 1-\frac{b-\yy}{2},
\end{equation}
and the integral is \Cb{along} an arbitrary simple closed contour which contains $0$ inside and lies in an annulus $0<|\zz|<\rr_0$.


\end{thm}
\begin{proof}
We consider an equivalent model: the TASEP on $\conf_N(L)$. The configuration space $\conf_N(L)$ is defined by
\begin{equation}
\conf_N(L) =\{(x_1,x_2,\cdots,x_N)\in\intZ^N: x_1<x_2<\cdots<x_N<x_1+L\}.
\end{equation}
The equivalence between TASEP on $\conf_N(L)$ and TASEP with $N$ particles on the ring of size $L$ is as \Cb{follows}: The ring TASEP can be obtained by projecting the particles in TASEP on $\conf_N(L)$ to a ring of size $L$\Cb{; On} the other hand, in the TASEP on a ring if we define $x_k$ to be the number of steps the $k$-th particle moved plus its initial location, then $(x_1,\cdots,x_N)\in\conf_N(L)$ and we obtain the TASEP on $\conf_N(L)$.  See \cite{Baik-Liu16} for more discussions on TASEP on a ring and its equivalent models. 

It is \Cb{not difficult} to see that the uniform initial condition for the TASEP of $N$ particles on a ring of size $L$ corresponds to the uniform initial condition in the following set
\begin{equation}
\label{eq:conf_initial}
\mathcal{Y}_N(L)=\{(y_1,y_2,\cdots,y_N)\in\intZ^N: -L+1\le y_1<y_2<\cdots<y_N\le 0\}
\end{equation}
in the system of TASEP on $\conf_N(L)$. Moreover, for any $Y\in\mathcal{Y}_N(L)$, we have the following relation between two models
\footnote{We first consider the case when $1\le \yy \le L$. In this case, $h_t(\yy)=2J_\yy(t)+h_0(\yy)=2J_\yy(t)+\yy-2\sum_{j=1}^\yy\eta_j(0)$. Therefore $h_t(\yy) \ge b$ if and only if $J_\yy(t) -\sum_{j=1}^\yy \eta_j(0)\ge (b-\yy)/2$, which is further equivalent to $x_{k'}(t)\ge a$. The case when $\yy\ge L+1$ or $\yy\le 0$ follows immediately from the fact that $h_t(\yy)=h_t(\yy-L)+(L-2N)$.}
\begin{equation}
\label{eq:aux_2016_08_29_07}
\prob\left(h_t(\yy) \ge b \mbox{ in TASEP on the ring with initial configuration }
	Y\right)= \prob_Y\left(x_{k'}(t)\ge a \right)
\end{equation}
\Cr{where the notation $\prob_Y$ denotes the probability of TASEP on $\conf_N(L)$ with initial configuration $Y\in\mathcal{Y}_N(L)$, and $x_{k'}(t)$ denotes the location of the $k'$-th particle at time $t$. The relation}~\eqref{eq:aux_2016_08_29_07}\Cr{ interprets the distribution function of $h_t(\yy)$ (for TASEP on a ring) as that of particle location $x_{k'}(t)$ (for TASEP on $\conf_N(L)$) at time $t$. The parameters $\yy$, $b$ on the left hand side of}~\eqref{eq:aux_2016_08_29_07}\Cr{ could be arbitrary integers satisfying $b\equiv \yy\pmod{2}$, and $k'$, $a$ on the right hand side are determined by}
\begin{equation}
\label{eq:relation_a_b}
\begin{split}
k' &= N\left[\frac{b-\yy-2}{2N}\right]+N+1-\frac{b-\yy}{2},\\
a &= L\left[\frac{b-\yy-2}{2N}\right] +\yy+1.
\end{split}
\end{equation}
\Cr{Here the notation $[y]$ denotes the integer part of $y$, i.e., the largest integer that is less than or equal to $y$.}
\Cr{From the above formula}~\eqref{eq:relation_a_b} \Cr{it is easy to see that $1\le k'\le N$. Hence $x_{k'}(t)$ is well defined in TASEP on $\conf_N(L)$.}

Now we sum over all possible initial configurations $Y\in\mathcal{Y}_N(L)$, each of which has probability $\frac{1}{{L\choose N}}$. \Cb{We} obtain
\begin{equation}
\label{eq:aux_2016_08_29_09}
\prob(h_t(\yy) \ge b) 
=\frac{1}{{L\choose N}}\sum_{Y\in\mathcal{Y}_N(L)}\prob_Y(x_{k'}(t)\ge a).
\end{equation}

On the other hand, the one point distribution function for TASEP on $\conf_N(L)$ with arbitrary initial condition $Y\in\conf_N(L)$ was obtained in \cite{Baik-Liu16} (see Proposition 6.1). More explicitly, we have
\begin{equation}
\label{eq:aux_2016_08_29_08}
\begin{split}
& \prob_Y(x_{k'}(t)\ge a)\\
&=\frac{(-1)^{(k'-1)(N+1)}}{2\pi\ii}\oint
\det\left[ \frac{1}{L}\sum_{w\in\roots_{\zz}}\frac{w^{j-i-k'+1}(w+1)^{y_j-j-a+k'+1}e^{tw}}{w+\rho}\right]_{i,j=1}^N
 \frac{\dd \zz}{\zz^{1-(k'-1)L}},
 \end{split}
 \end{equation}
 where the integral is \Cb{along} any simple closed contour with $0$ inside.
To proceed, we need the following two lemmas.
\begin{lm}
\label{lm:summation_initial_conf}
Suppose $w_1,w_2,\cdots,w_N\in\roots_{\zz}$, then we have
\begin{equation}
\label{eq:summation_initial_conf}
\sum_{Y\in\mathcal{Y}_N(L)}\det\left[w_i^j(w_i+1)^{y_j-j}\right]_{i,j=1}^N 
= \det\left[ w_i^{j-1}(w_i+1)^{-N+1}\right]_{i,j=1}^N-(-1)^{N-1}\zz^{-L}\det\left[ w_i^{j}(w_i+1)^{-N}\right]_{i,j=1}^N.
\end{equation}
\end{lm}
\begin{lm}
\label{lm:Fredholm_representation}(Theorem 7.2 in \cite{Baik-Liu16}) Suppose $\zz$ is in the annulus $0<|\zz|<\rr_0$ as in~\eqref{eq:r_bounds}. For any integer $k$, we have the following identity\footnote{The identity in \cite{Baik-Liu16} includes an integral over $\zz$. However, the proof is still valid if we drop the integral in both sides.}
\begin{equation}
\label{eq:Fredholm_identity}
(-1)^{(k-1)(N+1)}\zz^{(k+N-1)L}\det\left[ \frac{1}{L}\sum_{w\in\roots_{\zz}}\frac{w^{j-i-k-N+1}(w+1)^{-\yy+k}e^{tw}}{w+\rho}\right]_{i,j=1}^N
=\consts(\zz;k,\yy+1)\cdot \det\left(I+K_{\zz;k,\yy+1}^{(2)}\right).
\end{equation}
\end{lm}

We first assume Lemma~\ref{lm:summation_initial_conf} is true. By inserting~\eqref{eq:aux_2016_08_29_08} to~\eqref{eq:aux_2016_08_29_09} and then applying Lemma~\ref{lm:summation_initial_conf}, we have
\begin{equation}
\prob\left(h_t(\yy)\ge b\right)=\frac{1}{{L\choose N}}\Delta_{k'}\frac{(-1)^{(k'-2)(N+1)}}{2\pi\ii}\oint\det\left[\frac{1}{L}\sum_{w\in\roots_\zz} \frac{w^{j-i-k'+1}(w+1)^{-N-a+k'+1}e^{tw}}{w+\rho}\right]_{i,j=1}^N\frac{\dd\zz}{\zz^{1-(k'-2)L}}.
\end{equation}
Note that by using~\eqref{eq:aux_2016_10_05_01} this expression is invariant under the following changes: $a\to a-L$ and $k'\to k'-N$, therefore we can replace $a$ by $a-L\left[\frac{b-y-2}{2N}\right]=\yy+1$ and $k'$ by $k'-N\left[\frac{b-y-2}{2N}\right]=k+N$. The above equation equals to
\begin{equation}
\label{eq:aux_2016_08_29_10}
\frac{1}{{L\choose N}}\Delta_k \frac{(-1)^{(k+N-2)(N+1)}}{2\pi \ii}\oint \det\left[\frac{1}{L}\sum_{w\in\roots_\zz} \frac{w^{j-i-k-N+1}(w+1)^{-\yy+k}e^{tw}}{w+\rho}\right]_{i,j=1}^N\frac{\dd\zz}{\zz^{1-(k+N-2)L}}.
\end{equation}
By restricting $\zz$ in the annulus $0<|z|<\rr_0$ and applying Lemma~\ref{lm:Fredholm_representation} we immediately obtain~\eqref{eq:height_fluctuation}.


\vspace{0.2cm}

It remains to prove Lemma~\ref{lm:summation_initial_conf}.

We take the sum over $Y\in\mathcal{Y}_N(L)$ in the following order: $y_N$, $y_{N-1}, \cdots, y_1$. \Cb{Obviously}, the summation over $Y\in\mathcal{Y}_N(L)$ is equivalent to that over $y_j$: $y_{j-1}+1\le y_j \le j-N$ recurrently for $j=N,\cdots, 2$ and finally $-L+1\le y_1\le 1-N$. Note $y_j$ only appears in the $j$-th column in the determinant on the left hand side of~\eqref{eq:summation_initial_conf}. Hence if we take the sum over all possible $y_j$, all other columns in the determinant do not change except the $j$-th one. Then for each $j=N,\cdots,2$, we have the following sum over $y_j$ on the $j$-th column
\begin{equation}
\sum_{y_j=y_{j-1}+1}^{j-N}w_i^j(w_i+1)^{y_j-j} = w_i^{j-1}(w_i+1)^{-N+1} - w_i^{j-1}(w_i+1)^{y_{j-1}-(j-1)},
\end{equation}
where the second term is the same as the $(i,j-1)$ entry thus the determinant does not change if we remove this term. After taking the sum over $y_N,\cdots,y_2$, we obtain a new determinant whose first column is the same as before, but the $j$-th column  is $w_i^{j-1}(w_i+1)^{-N+1}$ for all $j=N,\cdots,2$. Then we take the sum over $y_1$. Note that the bounds for $y_1$ are $-L+1$ and $1-N$. Therefore we have
\begin{equation}
\begin{split}
&\sum_{Y\in\mathcal{Y}}\det\left[w_i^j(w_i+1)^{y_j-j}\right]_{i,j=1}^N\\
=&\det\left[ w_i^{j-1}(w_i+1)^{-N+1}-\delta_1(j)w_i^{j-1}(w_i+1)^{-L}\right]_{i,j=1}^N\\
=&\det\left[ w_i^{j-1}(w_i+1)^{-N+1}\right]_{i,j=1}^N- \det\left[ w_i^{j-1}(w_i+1)^{-N+1-\delta_1(j)(L-N+1)}\right]_{i,j=1}^N,
\end{split}
\end{equation}
where \Cb{we used the linearity of the determinant on the first column in the second equation. The notation $\delta_1(j)$ denotes the delta function}. Comparing \Cb{the above equation} with~\eqref{eq:summation_initial_conf}, \Cb{we only need to} show
\begin{equation}
\det\left[ w_i^{j-1}(w_i+1)^{-N+1-\delta_1(j)(L-N+1)}\right]_{i,j=1}^N=(-1)^{N-1}\zz^{-L}\det\left[ w_i^{j}(w_i+1)^{-N}\right]_{i,j=1}^N.
\end{equation}By using the fact that $(w_i+1)^{L-N}w_i^N =\zz^L$ and then exchanging the columns, the above equation is further reduced to
\begin{equation}
\det\left[ w_i^{j}(w_i+1)^{-N+1-\delta_N(j)}\right]_{i,j=1}^N=\det\left[ w_i^{j}(w_i+1)^{-N}\right]_{i,j=1}^N,
\end{equation}
which follows from the simple identity 
\begin{equation}
\left[ w_i^{j}(w_i+1)^{-N+1-\delta_N(j)}\right]_{i,j=1}^N = \left[ w_i^{j}(w_i+1)^{-N}\right]_{i,j=1}^N \left[\delta_i(j)+\delta_i(j+1)\right]_{i,j=1}^N.
\end{equation}

\end{proof}

\section{Asymptotic analysis and proof of Theorem~\ref{thm:limit_height_uniform}}
\label{sec:asymptotics}

In this section, we focus on the asymptotics of the formula~\eqref{eq:height_fluctuation} and prove Theorem~\ref{thm:limit_height_uniform}. We will follow the framework in \cite{Baik-Liu16}, where they computed the asymptotics of two similar formulas, one of which contains exactly the same components $\consts(\zz;k,\yy)$, $K_{\zz;k,\yy}^{(2)}$ and $\det(I+K_{\zz;k,\yy}^{(2)})$ as in this paper. However, there are the following two differences:
\begin{enumerate}[(1)]
	\item\label{item:01} In \cite{Baik-Liu16}, the asymptotics of $\consts(\zz;k,\yy)$ and $\det(I+K_{\zz;k,\yy}^{(2)})$ was obtained with a special choice of parameters. More explicitly, the authors considered a case of discrete times $t$ and an order $O(L)$ parameter $k$. In this paper, we have a different setting of parameters, in which we let $t$ \Cr{go} to infinity continuously and $k$ \Cr{grow} as $O(t)$.
	\item\label{item:02} The formula~\eqref{eq:height_fluctuation} in this paper contains a new feature. Namely, we have \Cb{the} difference operator $\Delta_k$, which was not present in \cite{Baik-Liu16}. In the asymptotics, this $\Delta_k$, after appropriate scaling, converges to the differentiation with respect to $x$.
\end{enumerate} 


For \ref{item:01}, one can modify the calculations in \cite{Baik-Liu16} to the new parameters. However, in this paper we instead consider a more general setting \Cr{of the parameters and prove } that both $\consts(\zz;k,\yy)$ and $\det(I+K_{\zz;k,\yy}^{(2)})$ converge simultaneously \Cr{with this general setting}. \Cr{It turns out that all the choices of the parameters considered in} \cite{Baik-Liu16} \Cr{and Theorem}~\ref{thm:limit_height_uniform} \Cr{ in this paper are included in the general setting. See Section}~\ref{sec:parameter_setting}\Cr{ for details.} 

For \ref{item:02}, we need to find the asymptotics of $\Delta_k\consts(\zz;k,\yy)$ and $\Delta_k\det(I+K_{\zz;k,\yy}^{(2)})$. The first one can be obtained straightforwardly, while the second one requires a bound estimate (uniformly on $L$ and $\zz$) of each term in its expansion, which guarantees the convergence (uniformly on $\zz$) of $\Delta_k\det(I+K_{\zz;k,\yy}^{(2)})$.

\subsection{Setting of the parameters}
\label{sec:parameter_setting}
In this subsection, we list the following general setting of the parameters.

We suppose the density $\rho=\rho_L=N_L/L$ satisfies  $c_1<\rho_L<c_2$ for some fixed positive constants $c_1,c_2$. \Cr{We} assume
\begin{equation}
\label{eq:parameters_t}
t=t_L=\frac{\tau}{\sqrt{\rho_L(1-\rho_L)}} L^{3/2}+O(L),
\end{equation}
for some fixed constant $\tau>0$. Moreover, suppose  $\yy=\yy_L$ and $k=k_L$ are two integer sequences which are bounded uniformly by $O(L^{3/2})$ and satisfy
\begin{equation}
\label{eq:parameters_condition1}
\dist\left(\frac{\yy_L-(1-2\rho_L)t_L-\gamma L}{L},\, \intZ\right)=O(L^{-1/2}),
\end{equation}
and\footnote{Here we view $t_L,\ell_L$ and $k_L$ as parameters for convenience of our analysis. We can also view $t_L,\ell_L$ and $b=b_L$, the height of $h_t(\ell)$, as parameters. By using~\eqref{eq:aux_2016_09_04_01}, we find that~\eqref{eq:parameters_condition2} is equivalent to \begin{equation}
	\frac{b_L-(1-2\rho_L)\ell_L-2\rho_L(1-\rho_L)t_L}{-2\rho_L^{2/3}(1-\rho_L)^{2/3}t_L^{1/3}}=x+O(L^{-1/2}).
	\end{equation}}
\begin{equation}
\label{eq:parameters_condition2}
\frac{k_L+\rho_L(1-\rho_L)t_L-\rho_L\yy_L}{\rho_L^{2/3}(1-\rho_L)^{2/3}t_L^{1/3}}=x+O(L^{-1/2}),
\end{equation}
where $\gamma=2w\tau^{2/3}$ and $x$ are arbitrary fixed real constants, and the notation $\dist(u,\intZ)$ denotes the smallest distance between $u$ and all integers. 

\Cb{Recall that the asymptotics along the line $\yy_L=(1-2\rho_L)t_L+(\gamma+1) L$ is the same as that along $\yy_L=(1-2\rho_L)t_L+\gamma L$}. See~\eqref{eq:aux_2016_10_06_01} and its discussions. The condition~\eqref{eq:parameters_condition1} means that the points should be asymptotically on the $\yy_L=(1-2\rho_L)t_L+(\gamma+\intZ) L$ lines.

To understand the second condition~\eqref{eq:parameters_condition2}, we need to view $k_L$ (more precisely $k_L+N_L$) as the label of \Cb{the} particle which is at the given location $\yy_L$ at time $t_L$. First we extend the TASEP on a ring to a periodic TASEP on $\intZ$ by making infinitely many identical copies of the particles on each interval of length $L$. More precisely, we define $x_{k+N}(t)=x_{k}(t)+L$ for all $k$ and $t$. With this setting, the labels of particles are in $\intZ$ instead of $\{1,2,\cdots,N\}$. ~\eqref{eq:parameters_condition2} means the label of \Cb{the} particle located at \Cb{the} site $\yy_L$ at time $t_L$ is $\rho_L \yy_L-\rho_L(1-\rho_L)t_L$ at the leading order (more precisely $N+\rho_L \yy_L-\rho_L(1-\rho_L)t_L$ due to our choice of initial labeling: the label is asymptotically $N$ at site $0$ initially), plus an $O(t_L^{1/3})$ fluctuation term. The term $\rho_L\yy_L$ (assuming $\yy_L>0$, otherwise $-\yy_L$ instead) is asymptotically the number of particles  initially in the interval $[0,\yy_L]$, while $\rho_L(1-\rho_L)t_L$ is asymptotically the number of particles jumping through any given site during time $[0,t_L]$.

The above descriptions are in terms of stationary TASEP on a ring with uniform initial condition. However, recalling the discussions at the beginning of Section~\ref{sec:asymptotics}, the formula arising in step initial condition contains the same components  $\consts(\zz;k,\yy)$ and $\det(I+K_{\zz;k,\yy}^{(2)})$, whose asymptotics can be found within the same framework. Thus  the conditions~\eqref{eq:parameters_condition1} and~\eqref{eq:parameters_condition2} can also be interpreted similarly in terms of TASEP on a ring with step initial condition.

\vspace{0.3cm}
Now we consider three different choices of parameters satisfying~\eqref{eq:parameters_t},~\eqref{eq:parameters_condition1} and~\eqref{eq:parameters_condition2}. 

The first choice is \Cb{to} fix the label of particle $k_L$ and then let $\yy_L$ and $t_L$ go to infinity simultaneously. This choice corresponds to the case when an observer focuses on a tagged particle. Now we rewrite the conditions~\eqref{eq:parameters_condition1} and~\eqref{eq:parameters_condition2} as
\begin{equation}
\yy_L-(1-2\rho_L)t_L=\gamma L+jL+O(L^{1/2})
\end{equation}
and
\begin{equation}
\yy_L-(1-\rho_L)t_L=\rho^{-1}k_L-x{\rho_L^{-1/3}(1-\rho_L)^{2/3}}t_L^{1/3},
\end{equation}
where $j=j_L$ is an  integer sequence. These two equations imply that
\begin{equation}
t_L=\frac{L}{\rho_L}j+\frac{\gamma}{\rho_L}L-\frac{1}{\rho_L^2}k_L+O(L^{1/2}).
\end{equation}
Now we want $t_L$ growing as~\eqref{eq:parameters_t}. Hence $j$ grows as $\left[\tau \rho_L^{1/2}(1-\rho_L)^{-1/2}L^{1/2}\right]$. For simplification, we ignore the $O(L^{1/2})$ in $t_L$ and obtain
\begin{equation}
t_L=\frac{L}{\rho_L}\left[\frac{\tau\sqrt{\rho_L}}{\sqrt{1-\rho_L}}L^{1/2}\right]+\frac{\gamma}{\rho_L} L -\frac{1}{\rho_L^2}k_L,
\end{equation}
which is a time scaling of TASEP on a ring with step initial condition discussed in \cite{Baik-Liu16} (with their $k_L$ replaced by $k_L+N_L$). See Theorem 3.3 of \cite{Baik-Liu16}.

The second choice of parameters is \Cb{to} fix the location $\yy_L$ and let $k_L$ and $t_L$ \Cr{go} to infinity simultaneously. This choice corresponds to the case when an observer focuses on a fixed location. By \Cb{an argument similar to the previous case}, we find that $t_L$ can be expressed as
\begin{equation}
\label{eq:aux_2016_09_04_02}
t_L=\frac{L}{|1-2\rho_L|}\left[\frac{|1-2\rho_L|\tau}{\sqrt{\rho_L(1-\rho_L)}}L^{1/2}\right]-\frac{\gamma L}{1-2\rho_L}+\frac{\yy_L}{1-2\rho_L}
\end{equation}
when $\rho_L$ is of $O(1)$ distance to $1/2$, \Cr{and}
\begin{equation}
t_L=2\tau L^{3/2}
\end{equation}
\Cr{when $\rho_L=1/2$.} \Cr{Note that when $\rho_L=1/2$, the line $\yy_L=const$ which describes the observer's location in the space-time plane is also the characteristic line with a constant shift. Hence} this case is reduced to the next one, \Cr{which} we will discuss later. These scalings were discussed in \cite{Baik-Liu16}, see Theorem 3.4 of that paper.


The third choice of parameters is \Cb{to} fix the line $\yy_L-(1-2\rho_L)t_L=\gamma L$. This is exactly the choice we pick in Theorem~\ref{thm:limit_height_uniform}. It means that an observe moves along the direction of the characteristic line. In this case, the time parameter $t_L$ can grow continuously, and the location $\yy_L$ changes according to $\yy_L-(1-2\rho_L)t_L=\gamma L$. Finally the label of particle grows by the formula~\eqref{eq:parameters_condition2}. Note that in Theorem~\ref{thm:limit_height_uniform} we have the height $h_{t_L}(\yy_L)$ instead of the label of particles $k_L$, hence to check~\eqref{eq:parameters_condition2} one needs to use the relation $k_L=\frac{\yy_L-b_L}{2}+1$ in~\eqref{eq:aux_2016_09_04_01}. 


\vspace{0.3cm}
For \Cb{notational} convenience, we will \Cr{suppress} the subscript $L$ in the asymptotic analysis from the next subsection to the end of Section~\ref{sec:asymptotics}.

\subsection{Preliminaries: choice of integral contour and parameter-independent asymptotics}

In this subsection we follow the setting of \cite{Baik-Liu16} (see Section 8) and give the explicit choice of integral contour. We also give the limit of $\roots_{\zz,\LL}$ and $\roots_{\zz,\RR}$,  and asymptotics of some parameter-independent components in $\consts(\zz;k,\yy)$. These results are all included in \cite{Baik-Liu16}. \Cb{Hence} we do not provide details.

In~\eqref{eq:height_fluctuation}, we set
\begin{equation}
\label{eq:aux_2016_09_04_03}
\zz^L=(-1)^{N}\rr_0^Lz,
\end{equation}
where $z$ is along any given simple closed contour within the unit disk $|z|<1$ and with $0$ inside. Then~\eqref{eq:height_fluctuation} becomes
\begin{equation}
\label{eq:prob_formula}
\prob\left(h_{t}(\yy)\ge b\right)=\frac{(-1)^{N+1}}{{L\choose N}}\oint\zz^{-L} \Delta_{k}\left(\consts(\zz; k,\yy+1)\cdot \det\left(I + K_{\zz;k,\yy+1}^{(2)}\right)\right)\ddbar{z},
\end{equation}
here $\zz=\zz(z)$ is any \Cr{branch} determined by~\eqref{eq:aux_2016_09_04_03}. And it is easy to check the integrand above is invariant for $\zz\to\zz e^{2\pi\ii /L}$. \Cb{Therefore} the choice of $\zz$, provided it satisfies~\eqref{eq:aux_2016_09_04_03}, does not affect the integral.

We first consider the limits of the nodes sets $\roots_{\zz,\LL}$ and $\roots_{\zz,\RR}$ with $\zz$ scaled as~\eqref{eq:aux_2016_09_04_03}. It turns out that after rescaling these nodes sets converge to the sets $\inodes_{z,\LL}=\{\xi: e^{-\xi^2/2}=z,\Re\xi<0\}$ and $\inodes_{z,\RR}=\{\xi:e^{-\xi^2/2}=z,\Re\xi>0\}$ respectively. The explicit meaning of this convergence is described as below.

\begin{lm}
	\label{lm:limit_nodes}
	(Lemma 8.1 of \cite{Baik-Liu16}) Let $z$ be a fixed number satisfying $0<|z|<1$ and let $\epsilon$ be a real constant satisfying $0<\epsilon<1/2$. Set $\zz^L=(-1)^{N}\rr_0^Lz$ where $\rr_0=\rho^{\rho}(1-\rho)^{1-\rho}$. Define the map $\mathcal{M}_{L,\LL}$ from $\roots_{\zz,\LL}\cap\left\{w:|w+\rho|\le \rho\sqrt{1-\rho}N^{\epsilon/4-1/2}\right\}$ to $\inodes_{z,\LL}$ by
	\begin{equation}
	\mathcal{M}_{N,\LL}(w)=\xi, \quad\mbox{where }\xi\in\inodes_{z,\LL}\mbox{ and }\left|\xi-\frac{(w+\rho)N^{1/2}}{\rho\sqrt{1-\rho}}\right|\le N^{3\epsilon/4-1/2}\log N.	\end{equation}
	Then for large enough $N$ we have:
	\begin{enumerate}[(a)]
		\item $\mathcal{M}_{N,\LL}$ is well-defined.
		\item $\mathcal{M}_{N,\LL}$ is injective.
		\item The following relations hold:
		\begin{equation}
		\inodes_{z,\LL}^{(N^{\epsilon/4}-1)}\subseteq I(\mathcal{M}_{N,\LL})\subseteq \inodes_{z,\LL}^{(N^{\epsilon/4}+1)},
		\end{equation}
		where $I(\mathcal{M}_{N,\LL}):=\mathcal{M}_{N,\LL}\left(\roots_{\zz,\LL}\cap \{\zz:|\zz+\rho|\le \rho\sqrt{1-\rho}N^{\epsilon/4-1/2}\}\right)$, the image of the map $\mathcal{M}_{N,\LL}$, and $\inodes_{z,\LL}^{(c)}:=\inodes_{z,\LL}\cap\{\xi:|\xi|\le c\}$ for all $c>0$.
	\end{enumerate}
	If we define the mapping $\mathcal{M}_{N,\RR}$ in the same way but replace $\roots_{\zz,\LL}$ and $\inodes_{z,\LL}$ by $\roots_{\zz,\RR}$ and $\inodes_{z,\RR}$
	respectively, the same results hold for $\mathcal{M}_{N,\RR}$.
\end{lm}

Then we consider the limits of $q_{\zz,\LL}(w)$, $q_{\zz,\RR}(w)$ and the following function
\begin{equation}
\label{eq:aux_2016_09_04_04}
\mathcal{C}_{N,1}^{(2)}(\zz):=\frac{\prod_{u\in\roots_{\zz,\LL}}(-u)^N\prod_{v\in\roots_{\zz,\RR}}(v+1)^{L-N}}{\prod_{u\in\roots_{\zz,\LL}}\prod_{v\in\roots_{\zz,\RR}}(v-u)}.
\end{equation}
The first two functions arise from the kernel $K_{\zz;k,y+1}^{(2)}$, and the third function $\mathcal{C}_{N,1}^{(2)}(\zz)$ is part of $\consts(\zz;k,\yy)$. The limits of these three functions were obtained in \cite{Baik-Liu16} as below.
\begin{lm}
	\label{lm:limits_ind_ftns}(Lemma 8.2 of \cite{Baik-Liu16}) Suppose $\zz,z$ and $\epsilon$ satisfy the conditions in Lemma~\ref{lm:limit_nodes}.
	\begin{enumerate}[(a)]
		\item For \Cr{a} complex number $\xi=\xi_N$ satisfying $c\le |\xi|\le N^{\epsilon/4}$ with some positive constant $c$, set $w_N=w_N(\xi)=-\rho+\rho\sqrt{1-\rho}\xi N^{-1/2}$. Then for sufficiently large $N$
		\begin{equation}
		q_{\zz,\LL}(w_N)=(w_N+1)^{L-N}e^{\hftn_\LL(\xi,z)}(1+O(N^{\epsilon-1/2}\log N))
		\end{equation}
		if $\Re\xi>c$, where 
		\begin{equation}
		\label{eq:def_hftn_LL}
		\hftn_{\LL}(\xi,z):=-\frac{1}{\sqrt{2\pi}}\int_{-\infty}^{-\xi}\polylog_{1/2}\left(ze^{(\xi^2-y^2)/2}\right)\dd y.
		\end{equation}
		Similarly for sufficiently large $N$
		\begin{equation}
		q_{\zz,\RR}(w_N)=(-w_N)^{N}e^{\hftn_\RR(\xi,z)}(1+O(N^{\epsilon-1/2}\log N))
		\end{equation}
		if $\Re\xi<-c$, where
		\begin{equation}
		\label{eq:def_hftn_RR}
		\hftn_\RR(\xi,z):=-\frac{1}{\sqrt{2\pi}}\int_{-\infty}^{\xi}\polylog_{1/2}\left(ze^{(\xi^2-y^2)/2}\right)\dd y.
		\end{equation}
	\item For large enough $N$ we have
	\begin{equation}
	\label{eq:aux_2016_09_05_01}
	\mathcal{C}_{N,1}^{(2)}(\zz)=e^{2B(z)}\left(1+O(N^{\epsilon-1/2})\right),
	\end{equation}
	where $B(z)=\frac{1}{4\pi} \int_0^z \frac{(\polylog_{1/2}(y))^2}{y} \dd y$ is defined in~\eqref{eq:def_constant_B}.
	\end{enumerate}
\end{lm}

Finally, we need the expansions of two functions $q_{\zz}(w)$ and $\frac{L(w+\rho)}{w(w+1)}$ along the line $\Re w=-\rho$. These estimates were obtained in \cite{Baik-Liu16}, see (9.36) and (9.37) of that paper. Below we give a quick summary of these estimates. Write $w=-\rho+\rho\sqrt{1-\rho}\xi N^{-1/2}$, where $\xi\in\ii\realR$.
It is straightforward to check that when $|\xi|\le N^{\epsilon/4}$
\begin{equation}
\label{eq:aux_2016_09_05_12}
N\log(1-\sqrt{1-\rho}\xi N^{-1/2})+(L-N)\log\left(1+\frac{\rho}{\sqrt{1-\rho}}\xi N^{-1/2}\right)=
-\frac{1}{2}\xi^2+\frac{2\rho-1}{3\sqrt{1-\rho}}\xi^3 N^{-1/2}+O(N^{\epsilon-1}),
\end{equation}
here and below $\log$ denotes the natural logarithm function with the branch cut $\realR_{\le 0}$.

Together with~\eqref{eq:aux_2016_09_05_03} and~\eqref{eq:aux_2016_09_04_03}, we have for $|\xi|\le N^{\epsilon/4}$
\begin{equation}
\label{eq:aux_2016_09_05_04}
\begin{split}
\frac{q_{\zz}(w)}{\zz^L}&=z^{-1}\left(1-\sqrt{1-\rho}\xi N^{-1/2}\right)^{N}\left(1+\frac{\rho}{\sqrt{1-\rho}}\xi N^{-1/2}\right)^{L-N}-1\\
&=
\frac{e^{-\xi^2/2}-z}{z}\left(1+\frac{2\rho-1}{3\sqrt{1-\rho}}\frac{e^{-\xi^2/2}}{e^{-\xi^2/2}-z}\xi^3 N^{-1/2}+O(N^{\epsilon-1})\right).
\end{split}
\end{equation}
When $|\xi|>N^{\epsilon/4}$, it is easy to check that
\begin{equation}
\label{eq:aux_2016_09_05_06}
\left|\frac{q_\zz(w)}{\zz^L}\right|\ge e^{cN^{\epsilon/2}}
\end{equation}
for some positive constant $c$.

Similarly, for $|\xi|\le N^{\epsilon/4}$, we have
\begin{equation}
\label{eq:aux_2016_09_05_05}
\frac{L(w+\rho)}{w(w+1)}=-\frac{1}{\rho\sqrt{1-\rho}}\xi N^{1/2}\left(1+\frac{1-2\rho}{\sqrt{1-\rho}}\xi N^{-1/2}+O(N^{-1})\right).
\end{equation}

\subsection{Asymptotics of $\consts(\zz;k,\yy)$}

\label{sec:asymptotics_C}

As we discussed before, the asymptotics of $\consts(\zz;k,\yy)$ was obtained in~\cite{Baik-Liu16} with a specific choice of parameters. The idea is as following: write $\consts(\zz;k,\yy)$ as $\mathcal{C}_{N,1}^{(2)}(\zz)\cdot\mathcal{C}_{N,2}^{(2)}(\zz;k,\yy)$, where $\mathcal{C}_{N,1}^{(2)}(\zz)$ is defined in~\eqref{eq:aux_2016_09_04_04}  and
\begin{equation}
\label{eq:aux_2016_09_04_05}
\mathcal{C}_{N,2}^{(2)}(\zz;k,\yy):=\prod_{u\in\roots_{\zz,\LL}}(-u)^{k-1}\prod_{v\in\roots_{\zz,\RR}}(v+1)^{-\yy+k}e^{tv}.
\end{equation}
With the parameter setting in \cite{Baik-Liu16}, they obtained (see Lemma 8.7 in \cite{Baik-Liu16})
\begin{equation}
\label{eq:aux_2016_09_05_02}
\lim_{N\to\infty}\mathcal{C}_{N,2}^{(2)}(\zz;k,\yy)= e^{\tau^{1/3}xA_1(z)+\tau A_2(z)}\left(1+O(N^{\epsilon-1/2})\right),
\end{equation}
where $A_1(z) = -\frac{1}{\sqrt{2\pi}} \polylog_{3/2}(z)$ and $A_2(z) = -\frac{1}{\sqrt{2\pi}} \polylog_{5/2}(z)$ are defined in~\eqref{eq:def_constants_A}.
Together with~\eqref{eq:aux_2016_09_05_01} in Lemma~\ref{lm:limits_ind_ftns}, one has
\begin{equation}
\label{eq:aux_2016_09_05_09}
\consts(\zz;k,\yy)=e^{\tau^{1/3}x A_1(z)+\tau A_2(z)+2B(z)}\left(1+O(N^{\epsilon-1/2})\right).
\end{equation}

The goal of this subsection is to check the proof of~\eqref{eq:aux_2016_09_05_02} in \cite{Baik-Liu16} also works under the more general setting~\eqref{eq:parameters_t},~\eqref{eq:parameters_condition1} and~\eqref{eq:parameters_condition2}. Considering that the asymptotic analysis in \cite{Baik-Liu16} was focusing on a different case which corresponds to the flat initial condition and~\eqref{eq:aux_2016_09_05_02} appearing in the step case was only discussed briefly, and that some parts of the proof will be used in later discussions, we would like to go through the main steps of the proof of~\eqref{eq:aux_2016_09_05_02} with the more general settings in this paper. However, we will not discuss many details \Cb{of} the calculations unless they are necessary.

First we write the summation in $\log\mathcal{C}_{N,2}^{(2)}(\zz;k,\yy)$ as an integral. By using a residue computation, it is easy to see that
\begin{equation}
\label{eq:sum_to_integral}
\begin{split}
&(k-1)\sum_{u\in\roots_{\zz,\LL}}\log(-u)+\sum_{v\in\roots_{\zz,\RR}}\left((-\yy+k)\log(v+1)+tv\right)\\
=&L\zz^L\int_{-\rho-\ii\infty}^{-\rho+\ii\infty}(\gee_2(w)-\gee_2(-\rho))\frac{w+\rho}{w(w+1)q_\zz(w)}\ddbarr{w},
\end{split}
\end{equation}
where
\begin{equation}
\label{eq:aux_2016_09_05_11}
\gee_2(w) = (k-1) \log (-w) +(\yy-k)\log (w+1) -tw.
\end{equation}

Now we \Cb{change} variables $w=-\rho+\rho\sqrt{1-\rho}\xi N^{-1/2}$ where $\xi\in\ii\realR$. Recall~\eqref{eq:aux_2016_09_05_06}, it is sufficient to consider the integral over $|\xi|\le N^{\epsilon/4}$ \Cb{since the integral for $|\xi|>N^{\epsilon/4}$} is exponentially small $O(e^{-cN^{\epsilon/2}})$. With this restriction and the assumptions that $\yy, k$ are bounded by $O(L^{3/2})$, we have
\begin{equation}
\label{eq:aux_2016_09_05_07}
\begin{split}
&\gee_2(w) -\gee_2(-\rho)\\
=&\frac{-k+\rho\yy-\rho(1-\rho)t}{\sqrt{1-\rho}N^{1/2}}\xi +\frac{(2\rho-1)k-\rho^2\yy}{2(1-\rho)N}\xi^2+\frac{-(1-3\rho+3\rho^2)k+\rho^3\yy}{3(1-\rho)^{3/2}N^{3/2}}\xi^3 +O(N^{\epsilon-1/2}).
\end{split}
\end{equation}
For \Cr{notational} simplification we \Cb{write} the first three terms $a_1\xi+a_2\xi^2+a_3\xi^3$. By using the conditions~\eqref{eq:parameters_t}-\eqref{eq:parameters_condition2}, it is \Cb{direct} to see that
\begin{equation}
\label{eq:aux_2016_09_05_08}
a_1=-\tau^{1/3}x+O(N^{-1/2}), \ a_2=O(N^{1/2}), \ a_3= O(1), \ -\frac{2(1-2\rho)a_2}{\sqrt{1-\rho}N^{1/2}}+3a_3=\frac{\rho(-k+\rho\yy)}{\sqrt{1-\rho}N^{3/2}}=\tau +O(N^{-1/2}).
\end{equation}

Now by plugging~\eqref{eq:aux_2016_09_05_07},~\eqref{eq:aux_2016_09_05_04}, and~\eqref{eq:aux_2016_09_05_05} we obtain that~\eqref{eq:sum_to_integral} equals to an exponentially small term $O(e^{-cN^{\epsilon/2}})$ plus
\begin{equation}
\begin{split}
-\int_{-\ii N^{\epsilon/4}}^{\ii N^{\epsilon/4}}\frac{z(a_1\xi^2+a_2\xi^3+a_3\xi^4)}{e^{-\xi^2/2}-z}\left(1-\frac{2\rho-1}{3\sqrt{1-\rho}}\frac{e^{-\xi^2/2}}{e^{-\xi^2/2}-z}\xi^3 N^{-1/2}\right)\left(1+\frac{1-2\rho}{\sqrt{1-\rho}}\xi N^{-1/2}\right)\ddbarr{\xi}+O(N^{\epsilon-1/2}).
\end{split}
\end{equation}
By using the symmetry of the integral domain and integrating by parts, \Cb{we find} that the above quantity equals to
\begin{equation}
\label{eq:aux_2016_09_05_10}
\begin{split}
&-a_1\int_{-\ii N^{\epsilon/4}}^{\ii N^{\epsilon/4}}\frac{z\xi^2}{e^{-\xi^2/2}-z}\ddbarr{\xi}-\left(-\frac{2(1-2\rho)a_2}{3\sqrt{1-\rho}N^{1/2}}+a_3\right)\int_{-\ii N^{\epsilon/4}}^{\ii N^{\epsilon/4}}\frac{z\xi^4}{e^{-\xi^2/2}-z}\ddbarr{\xi}+O(N^{\epsilon-1/2})\\
=&-a_1A_1(z)+\left(-\frac{2(1-2\rho)a_2}{\sqrt{1-\rho}N^{1/2}}+3a_3\right)A_2(z)+O(N^{\epsilon-1/2})
\end{split}
\end{equation}
where $A_1(z)=-\frac{1}{\sqrt{2\pi}}\polylog_{3/2}(z)=\int_{\Re\xi=0}\frac{z\xi^2}{e^{-\xi^2/2}-z}\ddbarr{\xi}$ and $A_2(z)=-\frac{1}{\sqrt{2\pi}}\polylog_{3/2}(z)=-\frac13\int_{\Re\xi=0}\frac{z\xi^4}{e^{-\xi^2/2}-z}\ddbarr{\xi}$ are defined in~\eqref{eq:def_constants_A}. Now we insert~\eqref{eq:aux_2016_09_05_08} into the above equation, we obtain \Cb{that} the right hand side equals to $\tau^{1/3}xA_1(z)+\tau A_2(z)+O(N^{\epsilon-1/2})$. Combing with~\eqref{eq:sum_to_integral}, we have~\eqref{eq:aux_2016_09_05_02}.

\subsection{Asymptotics of $\Delta_k\consts(\zz;k,\yy)$}

By definition, we have
\begin{equation}
\Delta_k\consts(\zz;k,\yy) = \consts(\zz;k,\yy)\left(\prod_{u\in\roots_{\zz,\LL}}(-u)\prod_{v\in\roots_{\zz,\RR}}(v+1)-1\right).
\end{equation}
By applying~\eqref{eq:aux_2016_09_05_09} and the following Lemma, we obtain
\begin{equation}
\label{eq:aux_2016_09_06_10}
\Delta_k\consts(\zz;k) = \frac{A_1(z)}{\sqrt{1-\rho}N^{1/2}}e^{\tau^{1/3}x A_1(z)+\tau A_2(z) +2B(z)}\left(1+O(N^{\epsilon-1/2})\right),
\end{equation}
where $\epsilon$ is the same as in the previous subsection.
\begin{lm}
\label{lm:estimate_diff_consts}
For any fixed $\epsilon$ satisfying $0<\epsilon<1/2$, we have
\begin{equation}
\label{eq:estimate_diff_consts}
\sum_{u\in\roots_{\zz,\LL}}\log (-u) +\sum_{v\in\roots_{\zz,\RR}}\log(v+1) = \frac{A_1(z)}{\sqrt{1-\rho}N^{1/2}}\left(1+O(N^{\epsilon-1/2})\right).
\end{equation}
\end{lm}
\begin{proof}
	By a residue computation similar to~\eqref{eq:sum_to_integral}, we write the left hand side of~\eqref{eq:estimate_diff_consts} as 
	\begin{equation}
	L\zz^L\int_{-\rho-\ii\infty}^{-\rho+\ii\infty}(\log(w/(-\rho))-\log((w+1)/(1-\rho)))\frac{w+\rho}{w(w+1)q_\zz(w)}\ddbarr{w}.
	\end{equation}
	The rest of the proof is similar to~\eqref{eq:aux_2016_09_05_10} but much easier. We omit the details.
\end{proof}

\subsection{Asymptotics of $\det(I+K_{\zz;k,\yy}^{(2)})$}
\label{sec:asymptotics_determinant}

Similar to $\consts(\zz;k,\yy)$, the asymptotics of  $\det(I+K_{\zz;k,\yy}^{(2)})$ was obtained in \cite{Baik-Liu16} with a special setting of parameters. The argument can \Cb{be applied} here for the general settings by a modification. Below we only provide the main steps and omit the details.

By using the property that $w^N(w+1)^{L-N}=\zz^L$ for arbitrary $w\in\roots_{\zz}$, we rewrite the determinant as $\det(I+ \tilde K_{\zz;k,\yy}^{(2)})$ with the kernel
\begin{equation}
\label{eq:aux_2016_09_05_15}
\tilde K_{\zz;k,\yy}^{(2)}(u_1,u_2) = h_2(u_1) \sum_{v\in\roots_{\zz,\RR}} \frac{1}{(u_1 -v)(u_2 -v) h_2(v)}
\end{equation}
where
\begin{equation}
h_2(w)= h_{2;k,\yy}(w)=\begin{dcases}
		\frac{g_2(w)}{w+\rho} \frac{q_{\zz,\RR}(w)^2}{w^{2N}},& w\in\roots_{\zz,\LL},\\
		\frac{g_2(w)}{w+\rho} \frac{q'_{\zz,\RR}(w)^2}{w^{2N}}, & w\in\roots_{\zz,\RR},
        \end{dcases}
\end{equation}
with 
\begin{equation}
\label{eq:def_g_2}
g_2(w) =g_{2;k,\yy}(w)=\frac{\tilde g_2(w)}{\tilde g_2(-\rho)} \frac{w^{jN}(w+1)^{j(L-N)}}{(-\rho)^{jN}(-\rho+1)^{j(L-N)}}
\end{equation}
and
\begin{equation}
\tilde g_2(w) = \tilde g_{2;k,\yy}(w) =w^{-k+2}(w+1)^{-\yy+k+1}e^{tw}.
\end{equation}
Here $j=j_L$ in~\eqref{eq:def_g_2} is an integer sequence satisfying
\begin{equation}
\yy-(1-2\rho)t-\gamma L=jL+O(L^{1/2}).
\end{equation}
The existence of such $j$ is guaranteed by~\eqref{eq:parameters_condition1}. Moreover, since we assume $t$ and $\yy$ are both at most $O(L^{3/2})$, we have $j\le O(L^{1/2})$.

Now we consider the asymptotics of $h_2(w)$. Write $w=-\rho+\rho\sqrt{1-\rho}\xi N^{-1/2}$. Then we have
\begin{equation}
g_2(w)=e^{-\gee_2(w)+\gee_2(-\rho)}\frac{w(w+1)}{-\rho(-\rho+1)}\left(1-\sqrt{1-\rho}\xi N^{-1/2}\right)^{jN}\left(1+\frac{\rho}{\sqrt{1-\rho}}\xi N^{-1/2}\right)^{j(L-N)}
\end{equation}
where $\gee_2$ is defined in~\eqref{eq:aux_2016_09_05_11}. If we further assume $|\xi|\le N^{\epsilon/4}$, the asymptotics of $g_2(w)$ can be obtained by using~\eqref{eq:aux_2016_09_05_07} and~\eqref{eq:aux_2016_09_05_12}
\begin{equation}
g_2(w)=e^{b_1\xi+b_2\xi^2+b_3\xi^3}(1+O(N^{\epsilon-1/2})),
\end{equation}
where
\begin{equation}
\label{eq:aux_2016_09_05_13}
\begin{split}
b_1&=-a_1=\tau^{1/3}x+O(N^{-1/2}),\\
b_2&=-a_2-\frac{1}{2}j=\frac{1}{2}\gamma+\frac{(1-2\rho)(-\rho\yy+k+\rho(1-\rho)t)}{2\rho(1-\rho)L}=\frac{1}{2}\gamma+O(N^{-1/2}),\\
b_3&=-a_3+\frac{2\rho-1}{3\sqrt{1-\rho}}jN^{-1/2}=\frac{(1-3\rho+3\rho^2)(k-\rho\yy)+(2\rho-1)^2\rho(1-\rho)t+O(L)}{3\rho^{3/2}(1-\rho)^{3/2}L^{3/2}}=-\frac{\tau}{3}+O(N^{-1/2}).
\end{split}
\end{equation}
Here in the second and third equations of~\eqref{eq:aux_2016_09_05_13} we used the conditions~\eqref{eq:parameters_condition2} and~\eqref{eq:parameters_t}. Thus we have
\begin{equation}
g_2(w)=e^{\tau^{1/3}x\xi+\frac{\gamma}{2}\xi^2-\frac{\tau}{3}\xi^3}(1+O(N^{\epsilon-1/2})),
\end{equation}

Together with Lemma~\ref{lm:limits_ind_ftns} (a), we immediately obtain the asymptotics of $h_2(w)$ when $|w+\rho|\le \rho\sqrt{1-\rho}N^{\epsilon/4}$. For the case when $|w+\rho|>\rho\sqrt{1-\rho}N^{\epsilon/4}$, one can show that $h_2(w)$ decays on $w\in\roots_{\zz,\LL}$ and grows on $w\in\roots_{\zz,\RR}$ exponentially fast as $w\to\infty$. The proof is similar to the case discussed in \cite{Baik-Liu16} and we do not provide details. The explicit asymptotics is described in the following lemma, which was proved for the special parameters in \cite{Baik-Liu16}.

\begin{lm}
\label{lm:estimate_kernel}
(Lemma 8.8 of \cite{Baik-Liu16})
Let $\epsilon$ be a fixed constant satisfying $0<\epsilon<1/2$.
\begin{enumerate}[(a)]
\item When $u\in\roots_{\zz,\LL}$ and $|u+\rho| \le \rho\sqrt{1-\rho} N^{\epsilon/4-1/2}$, we have
\begin{equation}
\label{eq:aux_2016_04_20_07}
h_2(u) = \frac{N^{1/2}}{\rho\sqrt{1-\rho}\xi} 
		  e^{2\hftn_\RR(\xi,z)-\frac{1}{3}\tau \xi^3 + \tau^{1/3}x\xi +\frac12\gamma\xi^2}
		  (1+O(N^{\epsilon-1/2}\log N)),
\end{equation}
where $\xi=\frac{N^{1/2}(u+\rho)}{\rho\sqrt{1-\rho}}$ and $\hftn_\RR$ is defined by~\eqref{eq:def_hftn_RR}, and the error term $O(N^{\epsilon-1/2}\log N)$ in~\eqref{eq:aux_2016_04_20_07} is independent of $u$ or $\xi$.

\item When $v\in\roots_{\zz,\RR}$ and $|v+\rho| \le \rho\sqrt{1-\rho} N^{\epsilon/4-1/2}$, we have
\begin{equation}
\label{eq:aux_2016_04_26_01}
\frac{1}{h_2(v)} = \frac{\rho^3(1-\rho)^{3/2}}{\zeta N^{3/2}} 
		  e^{2\hftn_\LL(\zeta,z)+\frac{1}{3}\tau \zeta^3 - \tau^{1/3}x\zeta -\frac12\gamma\zeta^2}
		  (1+O(N^{\epsilon-1/2}\log N)),
\end{equation}
where $\zeta=\frac{N^{1/2}(v+\rho)}{\rho\sqrt{1-\rho}}$ and $\hftn_\LL$ is defined by~\eqref{eq:def_hftn_LL}, and the error term $O(N^{\epsilon-1/2}\log N)$ in~\eqref{eq:aux_2016_04_26_01} is independent of $v$ or $\zeta$.

\item When $w\in\roots_{\zz}$ and $|w+\rho|\ge\rho\sqrt{1-\rho}N^{\epsilon/4-1/2}$, we have
 \begin{equation}
 h_2(w)=O(e^{-CN^{3\epsilon/4}}), \qquad w\in\roots_{\zz,\LL}
 \end{equation}
 or
  \begin{equation}
  \frac1{h_2(w)}=O(e^{-CN^{3\epsilon/4}}), \qquad w\in\roots_{\zz,\RR}.
  \end{equation}
Here both error terms $O(e^{-CN^{3\epsilon/4}})$ are independent of $w$.
\end{enumerate}
\end{lm}

The Lemmas~\ref{lm:limit_nodes} and~\ref{lm:estimate_kernel} indicate the following result
\begin{equation}
\label{eq:estimate_determinant}
\lim_{n\to\infty}\det\left(I+K_{\zz;k,\yy}^{(2)}\right) = \det(I-\Ks_{z;\tau^{1/3}x}),
\end{equation}
where $\Ks_{z;x}$ is an operator on $\inodes_{z,\LL}$ as defined in~\eqref{eq:def_inf_kernel_step}\footnote{Note that $\hftn_\LL(\zeta,z)=\hftn_{\RR}(-\zeta,z)=-\sqrt{\frac{1}{2\pi}}\int_{-\infty}^{-\zeta} \polylog_{1/2}(e^{-\omega^2/2})\dd \omega$ for $\zeta\in\inodes_{z,\RR}$. }. A rigorous proof needs a uniform bound of the Fredholm determinant on the left hand side and an error control when we change the space from $\roots_{\zz}$ to $\inodes_z$, both of which \Cb{were} considered in~\cite{Baik-Liu16} for their choice of parameters. \Cb{Their argument also works for the general setting of parameters}. Therefore we omit the details.

\subsection{Asymptotics of $\Delta_k\det(I+K_{\zz;k}^{(2)})$}
\label{sec:asymptotics_Delta_determinant}

Similar to the previous subsection, we write $\Delta_k\det(I+K_{\zz;k,\yy}^{(2)})$ as $\Delta_k\det(I+\tilde K_{\zz;k,\yy}^{(2)})$. 

\Cb{We first} need the following two lemmas.
\begin{lm}
\label{lm:moments_convergent}
For any fixed positive integer $m$, we have
\begin{equation}
\label{eq:aux_2016_09_05_14}
\begin{split}
&\lim_{n\to\infty}\sqrt{1-\rho}N^{1/2}\sum_{u_1,\cdots,u_m\in\roots_{\zz,\LL}}\Delta_k\det\left[\tilde K_{\zz;k,\yy}^{(2)}(u_i,u_j)\right]_{i,j=1}^m \\
=& \sum_{\xi_1,\cdots,\xi_m\in\inodes_{z,\LL}}\left.\frac{\dd}{\dd y}\right|_{y=\tau^{1/3}x}\det\left[-\Ks_{z;y}(\xi_i,\xi_j)\right]_{i,j=1}^m.
\end{split}
\end{equation}
\end{lm}
\begin{lm}
\label{lm:uniform_bound_difference_determinant}
There exists some constants $C$ and $C'$ which do not depend on $z$, such that for all positive integer $m$ we have
\begin{equation}
\label{eq:aux_2016_09_06_05}
N^{1/2}\sum_{u_1,\cdots,u_m\in\roots_{\zz,\LL}}\left|\Delta_k \det\left[\tilde K_{\zz;k,\yy}^{(2)}(u_i,u_j)\right]_{i,j=1}^m\right| \le 2mC^m
\end{equation}
for all $N\ge C'$.
\end{lm}

We assume both lemmas hold. By using  the dominated convergence theorem and the two lemmas above, we have
\begin{equation}
\lim_{N\to\infty}\sqrt{1-\rho}N^{1/2}\Delta_k\det(I+\tilde K_{\zz;k,\yy}^{(2)}) =\sum_{m\ge 1} \frac{1}{m!}\left.\frac{\dd}{\dd y}\right|_{y=\tau^{1/3}x}\sum_{\xi_1,\cdots,\xi_m\in\inodes_{z,\LL}}\det\left[-\Ks_{z;y}(\xi_i,\xi_j)\right]_{i,j=1}^m.
\end{equation}
Moreover, the right hand side is uniformly bounded. This further implies $\left.\frac{\dd }{\dd y}\right|_{y=\tau^{1/3}x} \det(I-\Ks_{z;y})$ is well defined and uniformly bounded. The above result \Cr{can thus} be written as
\begin{equation}
\label{eq:estimate_difference_determinant}
\lim_{n\to\infty}\sqrt{1-\rho}N^{1/2}\Delta_k\det(I+\tilde K_{\zz;k,\yy}^{(2)}) 
= \left.\frac{\dd }{\dd y}\right|_{y=\tau^{1/3}x} \det(I-\Ks_{z;y})
\end{equation}
uniformly on $z$.

\vspace{0.3cm}
Now we prove Lemmas~\ref{lm:moments_convergent} and~\ref{lm:uniform_bound_difference_determinant}.

\begin{proof}[Proof of Lemma~\ref{lm:moments_convergent}]
	Recall the definition of $\tilde K_{\zz;k,\yy}^{(2)}$ in~\eqref{eq:aux_2016_09_05_15}. It is easy to check that
	\begin{equation}
	\label{eq:aux_2016_09_06_06}
	\begin{split}
	&\Delta_k\det\left[\tilde K_{\zz;k,\yy}^{(2)}(u_i,u_j)\right]_{i,j=1}^m\\
	=&\sum_{v_1,\cdots,v_m\in\roots_{\zz,\RR}}\Delta_k\det\left[ \frac{h_{2;k,\yy}(u_i)}{(u_i-v_i)(u_j-v_i)h_{2;k,\yy}(v_i)}\right]_{i,j=1}^m\\
	=&\sum_{v_1,\cdots,v_m\in\roots_{\zz,\RR}}\left(\prod_{i=1}^m\frac{(u_i+1)v_i}{(v_i+1)u_i}-1\right)\det\left[ \frac{h_{2;k,\yy}(u_i)}{(u_i-v_i)(u_j-v_i)h_{2;k,\yy}(v_i)}\right]_{i,j=1}^m.
	\end{split}
	\end{equation}
	Here we emphasize the parameters in the function $h_2(w)$ to avoid confusion.
	Hence we have
	\begin{equation}
	\label{eq:aux_2016_09_05_16}
	\begin{split}
	&\sqrt{1-\rho}N^{1/2}\sum_{u_1,\cdots,u_m\in\roots_{\zz,\LL}}\Delta_k\det\left[\tilde K_{\zz;k,\yy}^{(2)}(u_i,u_j)\right]_{i,j=1}^m\\
	=&\sum_{\substack{u_1,\cdots,u_m\in\roots_{\zz,\LL}\\ v_1,\cdots,v_m\in\roots_{\zz,\RR}}}\sqrt{1-\rho}N^{1/2}\left(\prod_{i=1}^m\frac{(u_i+1)v_i}{(v_i+1)u_i}-1\right)\det\left[ \frac{h_{2;k,\yy}(u_i)}{(u_i-v_i)(u_j-v_i)h_{2;k,\yy}(v_i)}\right]_{i,j=1}^m.
	\end{split}
	\end{equation}
	Note that there are only $O(L^{2m})$ terms in the summation since $|\roots_{\zz}|=L$, and when $|u_i+\rho|\ge \rho\sqrt{1-\rho}N^{\epsilon/4}$ or $|v_i+\rho|\ge \rho\sqrt{1-\rho}N^{\epsilon/4}$ for some $i$ the summand is exponentially small (see Lemma~\ref{lm:estimate_kernel}). Therefore we can restrict the summation on all $u_i$ and $v_i$'s of at most $\rho\sqrt{1-\rho}N^{\epsilon/4}$ distance to $-\rho$. We write $u_i=-\rho+\rho\sqrt{1-\rho}\xi_i N^{-1/2}$ and $v_i=-\rho+\rho\sqrt{1-\rho}\zeta_i N^{-1/2}$, where $|\xi_i|, |\zeta_i|\le N^{\epsilon/4}$. Then by applying Lemma~\ref{lm:estimate_kernel} we have
	\begin{equation}
	\label{eq:aux_2016_09_05_17}
		\begin{split}
		&\eqref{eq:aux_2016_09_05_16}\\
		=&\sum_{\substack{\xi_1,\cdots,\xi_m\\ \zeta_1,\cdots,\zeta_m}}\left(\sum_{i=1}^m(\xi_i-\zeta_i)+O(N^{\epsilon-1/2})\right)\det\left[\frac{e^{\phi_\RR(\xi_i)-\phi_\LL(\zeta_i)}}{\xi_i\zeta_i(\xi_i-\zeta_i)(\xi_j-\zeta_i)}+O(N^{\epsilon-1/2}\log N)\right]_{i,j=1}^m\\
		&+O(e^{-cN^{\epsilon/2}}),
		\end{split}
	\end{equation}
	where the summation is over all possible $\xi_i$ and $\zeta_i$ such that $|\xi_i|, |\zeta_i|\le N^{\epsilon/4}$ and $-\rho+\rho\sqrt{1-\rho}\xi_i N^{-1/2}\in\roots_{\zz,\LL}$ and $-\rho+\rho\sqrt{1-\rho}\zeta_i N^{-1/2}\in\roots_{\zz,\RR}$ for all $i=1,\cdots,m$. And
	\begin{equation}
	\begin{split}
	\phi_\RR(\xi):=2\hftn_{\RR}(\xi,z)-\frac{1}{3}\tau\xi^3+\frac{1}{2}\gamma\xi^2+\tau^{1/3}x\xi,\\
	\phi_{\LL}(\zeta):=-2\hftn_{\LL}(\zeta,z)-\frac{1}{3}\tau\zeta^3+\frac{1}{2}\gamma\zeta^2+\tau^{1/3}x\zeta,
	\end{split}
	\end{equation}
	for $\xi$ and $\zeta$ satisfying $\Re\xi<0$ and $\Re\zeta>0$.
	Recall that the error terms in~\eqref{eq:aux_2016_09_05_17} are all uniformly on $\xi_i$ and $\eta_i$ (see Lemma~\ref{lm:estimate_kernel}), and note that there are at most $O(N^{\epsilon/2})$ elements by Lemma~\ref{lm:limit_nodes} part (c). Therefore~\eqref{eq:aux_2016_09_05_17} equals to
	\begin{equation}
	\label{eq:aux_2016_09_05_18}
	\sum_{\substack{\xi_1,\cdots,\xi_m\\ \zeta_1,\cdots,\zeta_m}}\sum_{i=1}^m(\xi_i-\zeta_i)\det\left[\frac{e^{\phi_\RR(\xi_i)-\phi_\LL(\zeta_i)}}{\xi_i\zeta_i(\xi_i-\zeta_i)(\xi_j-\zeta_i)}\right]_{i,j=1}^m +O(N^{(m+2)\epsilon-1/2}).
	\end{equation}
	Now by using Lemma~\ref{lm:limit_nodes} we know that these $\xi_i$ and $\zeta_i$'s are chosen from a perturbation of $I(\mathcal{M}_{N,\LL})$ and $I(\mathcal{M}_{N,\RR})$, the images of $\mathcal{M}_{N,\LL}$ and $\mathcal{M}_{N,\RR}$ respectively. The perturbation size is uniformly bounded by $N^{3\epsilon/4-1/2}\log N$. Similar to the reasoning from~\eqref{eq:aux_2016_09_05_17} to~\eqref{eq:aux_2016_09_05_18}, we can replace~\eqref{eq:aux_2016_09_05_18} by
		\begin{equation}
		\label{eq:aux_2016_09_05_19}
		\sum_{\substack{\xi_1,\cdots,\xi_m\in I(\mathcal{M}_{N,\LL})\\ \zeta_1,\cdots,\zeta_m\in I(\mathcal{M}_{N,\RR})}}\sum_{i=1}^m(\xi_i-\zeta_i)\det\left[\frac{e^{\phi_\RR(\xi_i)-\phi_\LL(\zeta_i)}}{\xi_i\zeta_i(\xi_i-\zeta_i)(\xi_j-\zeta_i)}\right]_{i,j=1}^m +O(N^{(m+2)\epsilon-1/2}).
		\end{equation}
		If we choose $\epsilon$ small \Cb{enough} such that $(m+2)\epsilon<1/2$, then the above quantity converges to
		\begin{equation}
		\label{eq:aux_2016_09_05_20}
		\sum_{\substack{\xi_1,\cdots,\xi_m\in \inodes_{z,\LL}\\ \zeta_1,\cdots,\zeta_m\in  \inodes_{z,\RR}}}\sum_{i=1}^m(\xi_i-\zeta_i)\det\left[\frac{e^{\phi_\RR(\xi_i)-\phi_\LL(\zeta_i)}}{\xi_i\zeta_i(\xi_i-\zeta_i)(\xi_j-\zeta_i)}\right]_{i,j=1}^m.
		\end{equation}
		Finally we check that~\eqref{eq:aux_2016_09_05_20} equals to the right hand side of~\eqref{eq:aux_2016_09_05_14}. This follows from the facts that  $\hftn_\RR(\xi,z)=-\frac1{\sqrt{2{\pi}}}\int_{-\infty}^{\xi} \polylog_{1/2}(e^{-\omega^2/2})\dd \omega$ for all $\xi\in\inodes_{z,\LL}$, and  $\hftn_{\LL}(\zeta,z)=-\frac1{\sqrt{2{\pi}}}\int_{-\infty}^{-\zeta} \polylog_{1/2}(e^{-\omega^2/2})\dd \omega$ for all $\zeta\in\inodes_{z,\RR}$, and that $\inodes_{z,\RR}=-\inodes_{z,\LL}$.
\end{proof}
\begin{proof}[Proof of Lemma~\ref{lm:uniform_bound_difference_determinant}]
	\Cb{We first prove the following Claim.}
	\begin{claim}
		\label{claim:01}
		There exist a positive constant $C$ and $C'$ uniformly on $z$ such that
		\begin{equation}
		\sum_{u_1\in\roots_{\zz,\LL}}\sqrt{\sum_{u_2\in\roots_{\zz,\LL}}\left|A(u_1,u_2)\right|^2}\le C
		\end{equation}
		for all $N\ge C'$,
		where
		\begin{equation}
		A(u_1,u_2):=\sqrt{|h_2(u_1)h_2(u_2)|}E(u_1)E(u_2)\sum_{v\in\roots_{\zz,\RR}}\frac{|E(v)|^2}{|u_1-v||u_2-v||h_2(v)|}
		\end{equation}
		and
		\begin{equation}
		E(w):=1+\frac{\rho|w+1|}{(1-\rho)|w|}+\frac{(1-\rho)|w|}{\rho|w+1|}+N^{1/2}\left(\left|\frac{(1-\rho)w}{\rho(w+1)}+1\right|+\left|\frac{\rho(w+1)}{(1-\rho)w}+1\right|\right).
		\end{equation}
	\end{claim}
	\begin{proof}[Proof of Claim~\ref{claim:01}]
		Note that $E(w)$ is always positive and  bounded by $c_1N^{1/2}+c_2$ uniformly on $\roots_{\zz}$. On the other hand, $h_2(u)$ and $h_2(v)^{-1}$ are exponentially small when $u\in\roots_{\zz,\LL}$ , $v\in\roots_{\zz,\RR}$ are  of distance $\ge O(N^{\epsilon/4-1/2})$, see Lemma~\ref{lm:estimate_kernel} (c). Thus it is sufficient to prove the following inequality
		\begin{equation}
		\label{eq:aux_2016_09_06_01}
		\sum_{\substack{u_1\in\roots_{\zz,\LL}\\ |u_1+\rho|\le N^{\epsilon/4-1/2}}}\sqrt{\sum_{\substack{u_2\in\roots_{\zz,\LL}\\ |u_2+\rho|\le N^{\epsilon/4-1/2}}}|h_2(u_1)h_2(u_2)E(u_1)^2E(u_2)^2|\left(\sum_{\substack{v\in\roots_{\zz,\RR}\\ |v+\rho|\le N^{\epsilon/4-1/2}}}\frac{|E(v)|^2}{|u_1-v||u_2-v||h_2(v)|}\right)^2}\le C.
		\end{equation}
		On the other hand, it is easy to check that 
		\begin{equation}
		\label{eq:aux_2016_09_06_03}
		E(-\rho+\rho\sqrt{1-\rho}\xi N^{-1/2})=3+\frac{2}{\sqrt{1-\rho}}|\xi| +O(N^{\epsilon-1/2})\le 3+C_1|\xi| +O(N^{\epsilon-1/2}),
		\end{equation}
		uniformly for all $|\xi|\le N^{\epsilon/4}$, here $C_1$ is a constant independent of $N$ (recall that $\rho=\rho_L\in (c_1,c_2)$ depends on $L$). We denote 
		\begin{equation}
		\tilde E(-\rho+\rho\sqrt{1-\rho}\xi N^{-1/2}):= \mbox{ the right hand side of }~\eqref{eq:aux_2016_09_06_03}.
		\end{equation} 
		Then~\eqref{eq:aux_2016_09_06_01} is reduced to
		\begin{equation}
		\label{eq:aux_2016_09_06_02}
		\sum_{\substack{u_1\in\roots_{\zz,\LL}\\ |u_1+\rho|\le N^{\epsilon/4-1/2}}}\sqrt{\sum_{\substack{u_2\in\roots_{\zz,\LL}\\ |u_2+\rho|\le N^{\epsilon/4-1/2}}}|h_2(u_1)h_2(u_2)\tilde E(u_1)^2\tilde E(u_2)^2|\left(\sum_{\substack{v\in\roots_{\zz,\RR}\\ |v+\rho|\le N^{\epsilon/4-1/2}}}\frac{|\tilde E(v)|^2}{|u_1-v||u_2-v||h_2(v)|}\right)^2}\le C.
		\end{equation}
				Using Lemmas~\ref{lm:limit_nodes} and~\ref{lm:estimate_kernel}, \Cb{we see that} the left hand side of~\eqref{eq:aux_2016_09_06_02} converges to 
		\begin{equation}
		\label{eq:aux_2016_09_06_04}
		\sum_{\xi_1\in\inodes_{z,\LL}}\sqrt{\sum_{\xi_2\in\inodes_{z,\RR}}\left|\frac{e^{\tilde\phi_\RR(\xi_1)+\tilde\phi_\RR(\xi_2)}}{\xi_1\xi_2}\right|\left(\sum_{\zeta\in\inodes_{z,\RR}}\frac{|e^{-\tilde{\phi}_\LL(\zeta)}|}{|\xi_1-\zeta||\xi_2-\zeta|}\right)^2},
		\end{equation}
		as $N\to\infty$, where $\tilde{\phi}_\RR(\xi):=\phi_\RR(\xi)+2\log\left(3+C_1|\xi|\right)$ and $\tilde{\phi}_\LL(\zeta):=\phi_\LL(\zeta)-2\log\left(3+C_1|\xi|\right)$. The rigorous proof of this convergence is similar to that of Lemma~\ref{lm:moments_convergent} and hence we do not provide details. Also it is easy to see that~\eqref{eq:aux_2016_09_06_04} is finite. Therefore~\eqref{eq:aux_2016_09_06_02} holds for sufficiently large $N$.
	\end{proof}

    Now we prove Lemma~\ref{lm:uniform_bound_difference_determinant}. This idea is to express the summand on the left hand side of~\eqref{eq:aux_2016_09_06_05} as a sum of determinants $\det\left[A^{(n)}(u_i,u_j)\right]$ where $A^{(n)}$ has similar structure of $A$ in Claim~\ref{claim:01}, and then apply the Hadamard's inequality. 
    
    The first step is to write
    \begin{equation}
    \det\left[\tilde{K}_{\zz;k,\yy}^{(2)}(u_i,u_j)\right]_{i,j=1}^m=
    \sum_{v_1,\cdots,v_m\in\roots_{\zz,\RR}}\det\left[\frac{\sqrt{h_{2;k,\yy}(u_i)}\sqrt{h_{2;k,\yy}(u_j)}}{(u_i-v_i)(u_j-v_i)h_{2;k,\yy}(v_i)}\right]_{i,j=1}^m
    \end{equation}
    by using a conjugation, here $\sqrt{h_{2;k,\yy}(u_i)}$ is the square root function with any fixed branch cut. Denote
    \begin{equation}
    H_{k,\yy}(u,u';v)=\frac{\sqrt{h_{2;k,\yy}(u)}\sqrt{h_{2;k,\yy}(u')}}{(u-v)(u'-v)h_{2;k,\yy}(v)}.
    \end{equation}
    Similarly to~\eqref{eq:aux_2016_09_06_06}, we have
    \begin{equation}
    \label{eq:aux_2016_09_06_07}
    \begin{split}
    &N^{1/2}\Delta_k\det\left[\tilde K_{\zz;k,\yy}^{(2)}(u_i,u_j)\right]_{i,j=1}^m\\
    =&N^{1/2}\sum_{v_1,\cdots,v_m\in\roots_{\zz,\RR}}\left(\prod_{i=1}^m\frac{(u_i+1)v_i}{(v_i+1)u_i}-1\right)\det\left[H_{k,\yy}(u_i,u_j;v_i)\right]_{i,j=1}^m\\
    =&\sum_{n=1}^m\sum_{v_1,\cdots,v_m\in\roots_{\zz,\RR}}N^{1/2}\left(\frac{-\rho(u_n+1)}{(1-\rho)u_n}-1\right)\prod_{i=1}^{n-1}\frac{-\rho(u_i+1)}{(1-\rho)u_i}\det\left[\frac{(1-\rho)v_iH_{k,\yy}(u_i,u_j;v_i)}{-\rho(v_i+1)}\right]_{i,j=1}^m\\
    &+\sum_{n=1}^m\sum_{v_1,\cdots,v_m\in\roots_{\zz,\RR}}N^{1/2}\left(1-\frac{-\rho(v_n+1)}{(1-\rho)v_n}\right)\prod_{i=1}^{n-1}\frac{-\rho(v_i+1)}{(1-\rho)v_i}\det\left[\frac{(1-\rho)v_iH_{k,\yy}(u_i,u_j;v_i)}{-\rho(v_i+1)}\right]_{i,j=1}^m\\
    =&\sum_{n=1}^m\det\left[A^{(n)}(u_i,u_j)\right]_{i,j=1}^m+\sum_{n=1}^m\det\left[\tilde A^{(n)}(u_i,u_j)\right]_{i,j=1}^m,
    \end{split}
    \end{equation}
    where 
    \begin{equation}
    A^{(n)}(u_i,u_j)=\begin{dcases}
    \frac{-\rho(u_i+1)}{(1-\rho)u_i}\sum_{v\in\roots_{\zz,\RR}}\frac{(1-\rho)v H_{k,\yy}(u_i,u_j;v)}{-\rho(v+1)},& 1\le i\le n-1,\\
    N^{1/2}\left(\frac{-\rho(u_n+1)}{(1-\rho)u_n}-1\right)\sum_{v\in\roots_{\zz,\RR}}\frac{(1-\rho)v H_{k,\yy}(u_i,u_j;v)}{-\rho(v+1)},&  i=n,\\
    \sum_{v\in\roots_{\zz,\RR}}\frac{(1-\rho)v H_{k,\yy}(u_i,u_j;v)}{-\rho(v+1)},& n+1\le i\le m,
    \end{dcases}
    \end{equation}
    and
    \begin{equation}
    \tilde A^{(n)}(u_i,u_j)=\begin{dcases}
    \sum_{v\in\roots_{\zz,\RR}} H_{k,\yy}(u_i,u_j;v),& 1\le i\le n-1,\\
    \sum_{v\in\roots_{\zz,\RR}} N^{1/2}\left(\frac{(1-\rho)v}{-\rho(v+1)}-1\right)H_{k,\yy}(u_i,u_j;v),&  i=n,\\
    \sum_{v\in\roots_{\zz,\RR}}\frac{(1-\rho)v H_{k,\yy}(u_i,u_j;v)}{-\rho(v+1)},& n+1\le i\le m.
    \end{dcases}
    \end{equation}
    It is easy to check that $|A^{(n)}(u_i,u_j)|$ and $|\tilde A^{(n)}(u_i,u_j)|$ are bounded by $|A(u_i,u_j)|$ defined in the Claim~\ref{claim:01}. By Hadamard's inequality, we have
    \begin{equation}
    \left|\det\left[A^{(n)}(u_i,u_j)\right]_{i,j=1}^m\right|\le \prod_{i=1}^m\sqrt{\sum_{1\le j\le m}|A^{(n)}(u_i,u_j)|^2}\le \prod_{i=1}^m\sqrt{\sum_{u'\in \roots_{\zz,\LL}}|A(u_i,u')|^2}
    \end{equation}
    for all distinct $u_1,\cdots,u_m\in\roots_{\zz,\LL}$. As a result,
    \begin{equation}
    \label{eq:aux_2016_09_06_08}
    \begin{split}
    \sum_{\substack{u_1,\cdots,u_m\in\roots_{\zz,\LL}\\ \mbox{\scriptsize all distinct}}}\left|\det\left[A^{(n)}(u_i,u_j)\right]_{i,j=1}^m\right|&\le \sum_{u_1,\cdots,u_m\in\roots_{\zz,\LL}}\prod_{i=1}^m\sqrt{\sum_{u'\in \roots_{\zz,\LL}}|A(u_i,u')|^2}\\
    &=\left(\sum_{u\in\roots_{\zz,\LL}}\sqrt{\sum_{u'\in \roots_{\zz,\LL}}|A(u,u')|^2}\right)^m\le C^m
    \end{split}
    \end{equation}
    by the Claim~\ref{claim:01}. Similarly we have
    \begin{equation}
        \label{eq:aux_2016_09_06_09}
    \sum_{\substack{u_1,\cdots,u_m\in\roots_{\zz,\LL}\\ \mbox{\scriptsize all distinct}}}\left|\det\left[\tilde A^{(n)}(u_i,u_j)\right]_{i,j=1}^m\right|\le C^m.
    \end{equation}Also note that $\Delta_k \det\left[\tilde K_{\zz;k,\yy}^{(2)}(u_i,u_j)\right]_{i,j=1}^m=0$ if $u_i=u_j$ for some $1\le i<j\le m$. By combing~\eqref{eq:aux_2016_09_06_07},~\eqref{eq:aux_2016_09_06_08} and~\eqref{eq:aux_2016_09_06_09} we obtain~\eqref{eq:aux_2016_09_06_05}.
\end{proof}

\subsection{Proof of Theorem~\ref{thm:limit_height_uniform}}
\label{sec:proof_Theorem_height}

Now we prove Theorem~\ref{thm:limit_height_uniform}(a). By using the estimates~\eqref{eq:aux_2016_09_05_09},~\eqref{eq:aux_2016_09_06_10},~\eqref{eq:estimate_determinant} and~\eqref{eq:estimate_difference_determinant}, we obtain
\begin{equation}
\label{eq:aux_2016_04_26_02}
\begin{split}
&\lim_{n\to\infty}\sqrt{1-\rho}N^{1/2}\Delta_k\left( \consts(\zz; k)\cdot \det\left(I + K_{\zz;k}^{(2)}\right)\right)\\
& = \left.\frac{\dd}{\dd y}\right|_{y=\tau^{1/3}x} \left(e^{yA_1(z)+\tau A_2(z) +2B(z)}\det\left(I -\Ks_{z;y}\right)\right).
\end{split}
\end{equation}
Furthermore, by the discussions below~\eqref{eq:estimate_determinant} and Lemma~\ref{lm:uniform_bound_difference_determinant}, we know the left hand side of~\eqref{eq:aux_2016_04_26_02} is uniformly bounded on $z$. 

On the other hand, by using the sterling's formula and~\eqref{eq:aux_2016_09_04_03}, we obtain
\begin{equation}
\label{eq:aux_2016_09_19_06}
\frac{(-1)^N}{{L\choose N}}\frac{1}{\sqrt{1-\rho} N^{1/2}\zz^L} = \frac{ N!(L-N)!}{L!\rho^N(1-\rho)^{L-N}\sqrt{1-\rho} N^{1/2}z} 
= \frac{\sqrt{2\pi} }{z}(1+O(N^{-1})).
\end{equation}

Theorem~\ref{thm:limit_height_uniform}  follows immediately by inserting the above two estimates into~\eqref{eq:prob_formula}.

\appendix
\section{Tail bound of the limiting distribution}
\label{sec:estimates}

In this appendix, we give some tail bounds related to the function $\FU$. These estimates are not optimal, however, they are sufficient to show that (1) $\FU(x;\tau,\gamma)$ is a distribution function, and (2) the $n$-th moments of $\frac{h_{t_L}(\yy_L)-(1-2\rho_L)\yy_L-2\rho_L(1-\rho_L)t_L}{-2\rho_L^{1/2}(1-\rho_L)^{1/2}L^{1/2}}$ converges to that of $\int_{\realR}x^n\dd \FU(x;\tau,\gamma)$ for any finite $n$ as $L\to\infty$, here $\rho_L,\yy_L$ and $t_L$ are defined in Theorem~\ref{thm:limit_height_uniform}. The second statement follows in the same way as Theorem 1 in \cite{Baik-Ferrari-Peche14}.

For simplification we only consider the case when $\tau=1$. For other values of $\tau$, the statements and proofs are the same (with different constants).

Define
\begin{equation}
F^{(L)}_U(x):=\prob\left(\frac{h_{t_L}(\yy_L)-(1-2\rho_L)\yy_L-2\rho_L(1-\rho_L)t_L}{-2\rho_L^{2/3}(1-\rho_L)^{2/3}t_L^{1/3}} \le x\right)
\end{equation}
and
\begin{equation}
G^{(L)}_U(x):=\frac{(-1)^{N_L+1}}{{L\choose N_L}}\oint  \consts(\zz; k_L,\yy_L+1)\cdot \det\left(I + K_{\zz;k_L,\yy_L+1}^{(2)}\right) \frac{\dd \zz}{2\pi \ii \zz^{L+1}}
\end{equation}
where the parameters and notations are the same as in Theorem~\ref{thm:limit_height_uniform}, and we \Cr{suppress} the parameters $\tau=1$ and $\gamma$ in the indices for simplification, and
\begin{equation}
k_L=1+ \rho_L\yy_L- \rho_L(1-\rho_L)t_L+x\rho_L^{2/3}(1-\rho_L)^{2/3}t_L^{1/3}.
\end{equation}
By using Theorem~\ref{thm:current_fluctuation}, it is easy to check 
\begin{equation}
F^{(L)}_U(x)=\frac{t_L^{-1/3}}{\rho_L^{2/3}(1-\rho_L)^{2/3}}\frac{\dd }{\dd \tilde x}G^{(L)}_U(\tilde x)
\end{equation}
where $\tilde x$ is an point satisfying $\tilde x=x+O(t_L^{-1/3})$.

\begin{prop} (Left tail bound of $\FU^{(L)}$)
	\label{prop:properties_FU}
	There exist constants $\alpha>0$, $c>0, C>0$ and $C'>0$, such that
	\begin{equation}
	\FU^{(L)}(x)\le e^{-c|x|^\alpha}
	\end{equation}
	for all $x\le -C$ and $L\ge C'|x|$.
\end{prop}

\begin{prop}
	(Right tail bound of $G^{(L)}_U$)
	\label{prop:properties_GU}
		There exist constants $\alpha>0$, $c>0, C>0$ and $C'>0$, such that
		\begin{equation}
		\label{eq:aux_2016_10_06_02}
		\left|x+1-\frac{t_L^{-1/3}}{\rho_L^{2/3}(1-\rho_L)^{2/3}}G_U^{(L)}(x)\right|\le e^{-cx^\alpha}
		\end{equation}
		for all $x\ge C$ and $L\ge C'x^6$.\footnote{For general $\tau$, the term $x+1$ in~\eqref{eq:aux_2016_10_06_02} should be replaced by $x+\tau$.}
\end{prop}

Although we use the same notations of constants $\alpha, c, C$ and $C'$ in the above two propositions, their values are not the same.

We also remark that these two propositions are analogous to Proposition 1 and 2 in \cite{Baik-Ferrari-Peche14}.

\subsection{Proof of Proposition~\ref{prop:properties_FU}}

The idea of the proof is to map the periodic TASEP to the periodic directed last passage percolation (DLPP). The relation was discussed in \cite{Baik-Liu16} and \cite{Baik-Liu16b} and we refer the readers to Section 3.1 of \cite{Baik-Liu16b} for more details. Here we  give a brief description. 

We first introduce the periodic TASEP. This is equivalent to TASEP on $\conf_N(L)$ except we have infinitely many copies of particles, which satisfy $x_{k+N}(t)=x_{k}(t)+L$ for all $k\in\intZ$. 

Similarly to the mapping between the infinite TASEP and usual DLPP, see \cite{Johansson00}, there is a mapping from periodic TASEP to periodic DLPP described as following: Let $\mr v=(L-N,-N)$ be the period vector, and $\Gamma$ be a lattice path with lower left corners $(i+x_{N+1-i}(0),i)$ for $i\in\intZ$. It is easy to check that $\Gamma$ is invariant if translated by $\mr v$.  Let $w(\mr q)$ be random exponential variables with parameter $1$ for all lattice points $\mr q$ which are on the upper right side of $\Gamma$. We require $w(\mr q)=w(\mr q+\mr v)$ for all $\mr q$. Except for this restriction, all $w(\mr q)$ are independent. We then define
\begin{equation}
H_{\mr p}(\mr q)=\max_{\pi}\sum_{\mr r\in\pi}w(\mr r)
\end{equation}
where the maximum is over all the possible up/right lattice paths from $\mr p$ to $\mr q$. We also define
\begin{equation}
H_{\Gamma}(\mr q)=\max_{\mr p}H_{\mr p}(\mr q).
\end{equation}

Now we are ready to introduce the relation between particle location in periodic \Cr{TASEP} and last passage time in periodic DLPP, see (3.7) in \cite{Baik-Liu16b},
\begin{equation}
\label{eq:aux_2016_09_19_01}
\prob_{\mr v}\left(x_k(t)\ge a \right)=\prob_{\mr v}\left(H_\Gamma(N+a-k,N+1-k)\le t\right),
\end{equation}
where we use the notation $\prob_{\mr v}$ to denote the probability functions in periodic TASEP and the equivalent periodic DLPP model. Using~\eqref{eq:aux_2016_09_19_01} and the relation between height function $h_{t}(\yy_L)$ and the particle location $x_k(t)$, see~\eqref{eq:aux_2016_08_29_07}, it is straightforward to show the following
\begin{equation}
\FU^{(L)}(x)=\prob_{\mr v}\left(H_{\Gamma}(\mr q)\le t_L\right)
\end{equation}
where $\mr q=(\mr q_1,\mr q_2)$ with
\begin{equation}
\label{eq:aux_2016_09_18_05}
\begin{split}
\mr q_1&=(1-\rho_L)^2t_L+\gamma(1-\rho_L)L-x\rho_L^{2/3}(1-\rho_L)^{2/3}t_L^{1/3},\\
\mr q_2&=\rho_L^2t_L-\gamma\rho_LL-x\rho_L^{2/3}(1-\rho_L)^{2/3}t_L^{1/3}.
\end{split}
\end{equation}
The rest of this section is to show that there exist constants $\alpha>0$, $c>0$, $C>0$, and  $C'>0$, such that 
\begin{equation}
\label{eq:aux_2016_09_18_02}
\prob_{\mr v}\left(H_{\Gamma}(\mr q)\le t_L\right) \le e^{-c|x|^\alpha}
\end{equation}
for all $x<-C$ and $L\ge C'|x|$. Then Proposition~\ref{prop:properties_FU} follows immediately.

\vspace{0.2cm}
The idea to prove~\eqref{eq:aux_2016_09_18_02} is to compare the periodic DLPP with the usual DLPP. This idea was applied in \cite{Baik-Liu16b} for periodic TASEP in sub-relaxation time scale. In the case we consider in this paper, we need a relaxation time analogous of the argument. We first introduce some known results on DLPP model. The probability space for DLPP is that all the lattice points $\mr q$ are associated with an \iid exponential random variable $w(\mr q)$, we use $\prob$ to denote the probability associated to this space. Similarly to the periodic DLPP, we denote $G_{\mr p}(\mr q)$ the last passage time from $\mr p$ to $\mr q$, and $G_{\Lambda}(\mr q)$ the last passage time from the lattice path $\Lambda$ to $\mr q$. Finally we define $B(c_1,c_2):=\{\mr q=(\mr q_1,\mr q_2)\in\intZ_{\ge 0}^2; c_1\mr q_1\le \mr q_2\le c_2\mr q_1\}$ for arbitrary constants $c_1,c_2$ satisfying $0<c_1<c_2$. From now on we fix these two constants $c_1$ and $c_2$. It is known that \cite{Johansson00}
\begin{equation}
\lim_{\substack{|\mr q|\to\infty\\ \mr q\in B(c_1,c_2)}}\prob\left(\frac{G(\mr q)-d(\mr q)}{s(\mr q)}\le x\right)=\FGUE(x),
\end{equation}
where $d(\mr q)=(\sqrt{\mr q_1}+\sqrt{\mr q_2})^2$ and $s(\mr q)=(\mr q_1\mr q_2)^{-1/6}(\sqrt{\mr q_1}+\sqrt{\mr q_2})^{4/3}$.
The following tail estimate is also needed, which is due to \cite{Baik-Ben_Arous-Peche05,Baik-Ferrari-Peche14},
\begin{equation}
\label{eq:aux_2016_09_18_03}
\prob\left(\frac{G(\mr q)-d(\mr q)}{s(\mr q)}\le -y\right)\le e^{-c_3y}
\end{equation}
for sufficiently large $y\ge C_1$ and $\mr q\in B(c_1,c_2)$ satisfying $|\mr q|\ge C'_1$. Here $c_3$, $C_1$ and $C'_1$ are constants only depend on $c_1$ and $c_2$.  The last result in DLPP we need is \Cb{an estimate of} the transversal fluctuations. Define $B_{\overline{\mr p\mr q}}(y)$ to be the set of all lattice points $\mr r$ satisfying
\begin{equation}
\dist\left( \mr r, \overline{\mr p\mr q}\right)\le y|\mr q-\mr p|^{2/3},
\end{equation}
\Cr{where $\overline{\mr p\mr q}$ denotes the line passing through the two points $\mr p$ and $\mr q$, and $\dist\left( \mr r, \overline{\mr p\mr q}\right)$ denotes the distance between the point $\mr r$ and the line $\overline{\mr p\mr q}$.}
We also define $\pi_{\mr p}^{max}(\mr q)$ to be the maximal path from $\mr p$ to $\mr q$ in the usual DLPP.  The following transversal fluctuation estimate is currently known: There exist constants $c_4$, $C_2$ and $C'_2$ such that
\begin{equation}
\label{eq:aux_2016_09_18_04}
\prob\left(\pi_{\mr 0}^{max}(\mr q)\subseteq B_{\overline{\mr 0\mr q}}(y)\right)\ge 1-e^{-c_4y}
\end{equation}
for all $y\ge C_2$ and and $\mr q\in B(c_1,c_2)$ satisfying $|\mr q|\ge C'_2$. The analog of this estimate in Poissonian version of DLPP was obtained in \cite{Basu-Sidoravicius-Sly16} and their idea can be applied in the exponential case similarly. We hence do not provide a proof here, instead we refer the readers to a forthcoming paper \cite{Nejjar16} by Nejjar for more discussions.

\vspace{0.2cm}

Now we use~\eqref{eq:aux_2016_09_18_03} and~\eqref{eq:aux_2016_09_18_04} to prove~\eqref{eq:aux_2016_09_18_02}. We pick $k+1$ equidistant points $\mr 0=\mr q^{(0)}, \mr q^{(1)}\cdots, \mr q^{(k)}=\mr {q}$ on the line $\overline{\mr {0}\mr q}$ such that
\begin{equation}
\dist\left(\mr v, \overline{\mr 0\mr q}	\right)\ge C_2|\mr q^{(i+1)}-\mr q^{(i)}|^{2/3},\qquad i=0,1,\cdots, k-1,
\end{equation} 
here $k$ is some large parameter which will be decided later. Note that $\dist\left(\mr v, \overline{\mr 0\mr q}\right)=O(|\mr q|^{2/3})$, hence the above inequality is satisfied as long as $k$ is greater than certain constant.

Now note that $H_{\Gamma}(\mr q)\ge H_{(1,1)}(\mr q)=H_{\mr 0}(\mr q)+O(1)$ since $(1,1)$ is at the upper right side of to the initial contour $\Gamma$ by definition, and $H_{\mr 0}(\mr q)\ge \sum_{i=0}^{k-1}H_{\mr q^{(i)}}(\mr q^{(i+1)})$, therefore
\begin{equation}
\label{eq:aux_2016_09_18_06}
\prob_{\mr v}\left(H_{\Gamma}(\mr q)\le t_L\right)\le k\prob_{\mr v}\left(H_{\mr 0}(\mr q^{(1)})\le t_L/k\right).
\end{equation} 
On the other hand, by using~\eqref{eq:aux_2016_09_18_04} we know that
\begin{equation}
\label{eq:aux_2016_09_18_07}
\prob_{\mr v}\left(H_{\mr 0}(\mr q^{(1)})\le t_L/k\right) \le \prob \left(G_{\mr 0}(\mr q^{(1)})\le t_L/k\right) +e^{-c_4k^{2/3}|\mr q|^{-2/3}\dist(\mr v,\overline{\mr 0\mr q})}
\end{equation}
provided $|\mr q|\ge C'_2k$. Finally, by inserting~\eqref{eq:aux_2016_09_18_05} and then applying~\eqref{eq:aux_2016_09_18_03}, we have
\begin{equation}
\label{eq:aux_2016_09_18_08}
\prob \left(G_{\mr 0}(\mr q^{(1)})\le t_L/k\right) \le e^{-c_5k^{-2/3}|x|}
\end{equation}
provided $|\mr q|\ge C'_1k$ and $x<-C$, where $c_5$ and $C$ are constants. By combing~\eqref{eq:aux_2016_09_18_06},~\eqref{eq:aux_2016_09_18_07} and~\eqref{eq:aux_2016_09_18_08}, we obtain
\begin{equation}
\prob_{\mr v}\left(H_{\Gamma}(\mr q)\le t_L\right)\le ke^{-c_5k^{-2/3}|x|}+ke^{-c_4k^{2/3}|\mr q|^{-2/3}\dist(\mr v,\overline{\mr 0\mr q})}
\end{equation}
Finally we pick $k=|x|$ and~\eqref{eq:aux_2016_09_18_02} follows immediately.
\subsection{Proof of Proposition~\ref{prop:properties_GU}}

The proof is similar to that of Theorem~\ref{thm:limit_height_uniform} but we do not need to handle the difference operator. We only provide the main ideas here.

First we do the same change of variables as in~\eqref{eq:prob_formula} and write

\begin{equation}
G_U^{(L)}(x)=\frac{(-1)^{N+1}}{{L\choose N}}\oint\zz^{-L} \left(\consts(\zz; k,\yy+1)\cdot \det\left(I + K_{\zz;k,\yy+1}^{(2)}\right)\right)\ddbar{z}.
\end{equation}
Now we assume $x$ is large and  pick $z$ on the following circle
\begin{equation}
\label{eq:aux_2016_09_19_02}
|z|=e^{-x}.
\end{equation}
With this choice of $z$, by using a similar argument as in Section~\ref{sec:asymptotics_C} we have
\begin{equation}
\label{eq:aux_2016_09_19_08}
\consts(\zz; k,\yy+1)=e^{xA_1(z)+A_2(z)+2B(z)}(1+O(L^{-1/3}))=\left(1-\frac{1}{\sqrt{2\pi}}(x+1)z\right)(1+O(L^{-1/3}))+O(ze^{-cx})
\end{equation}
provided $L\gg x^6$. By tracking the error terms, the term $O(L^{-1/3})$ is analytic in $z$ and can be expressed as $c+c'z+O(z^2L^{-1/3})$ with $c,c'$ both bounded by $O(L^{-1/3})$.

 Similarly to Section~\ref{sec:asymptotics_determinant}, we write $\det\left(I + K_{\zz;k,\yy+1}^{(2)}\right)$ as $\det(I+ \tilde K_{\zz;k,\yy}^{(2)})$ whose kernel is defined in~\eqref{eq:aux_2016_09_05_15}. By a similar argument as Lemma~\ref{lm:estimate_kernel}, one can show that the kernel decays exponentially
\begin{equation}
\label{eq:aux_2016_09_19_03}
\left|\tilde K_{\zz;k,\yy}^{(2)}(\xi,\eta)\right|\le e^{-c\left(\Re \left(-\frac{1}{3}\xi^3+x\xi\right)+\left(-\frac{1}{3}\eta^3+x\eta\right)\right)}
\end{equation}
for all $\xi,\eta\in \inodes_{z,\LL}$ and sufficiently large $x$. Here $c>0$ is a constant. The heuristic argument is as following: Suppose $\xi=a+\ii b\in \inodes_{z,\LL}$ with $a<0$, then $a^2-b^2=2x$ by~\eqref{eq:aux_2016_09_19_02}. It is a direct to show that the leading term  in the exponent of $h(u)$ in Lemma~\ref{lm:estimate_kernel} (a) (after dropping the term $\frac12\gamma\xi^2$, whose real part is independent of $\xi$ and hence cancels with the counterpart from $1/h(v)$) is
\begin{equation}
\label{eq:aux_2016_09_19_04}
\Re\left(-\frac{1}{3}\xi^3+x\xi\right)=\frac{2}{3}a^3-xa\le \frac{1}{3}xa\le -\frac{2}{3}x^{3/2}\ll 0.
\end{equation}
Similar estimates for the \Cb{leading term in the exponent} of $1/h(v)$ in  Lemma~\ref{lm:estimate_kernel} (b) hold. Therefore we have~\eqref{eq:aux_2016_09_19_03}. Finally, by using~\eqref{eq:aux_2016_09_19_03} and~\eqref{eq:aux_2016_09_19_04}, it is a direct to prove that
\begin{equation}
\label{eq:aux_2016_09_19_07}
\det(I+ \tilde K_{\zz;k,\yy}^{(2)})= 1+ O(e^{-cx^{3/2}})
\end{equation}
for a different positive constant $c$. Since the above argument is similar to that in Section~\ref{sec:asymptotics_determinant}, we do not provide details.

Finally by combing~\eqref{eq:aux_2016_09_19_06},~\eqref{eq:aux_2016_09_19_08} and~\eqref{eq:aux_2016_09_19_07}, also noting that $\zz^L=(-1)^{N}\rr_0^Lz$, we obtain that
\begin{equation}
G_U^{(L)}(x)={\sqrt{\rho_L(1-\rho_L)}L^{1/2}}\left(x+1+O(e^{-cx})\right).
\end{equation}
Hence we obtain Proposition~\ref{prop:properties_GU}.


\begin{thebibliography}{10}
	
	\bibitem{Baik-Ben_Arous-Peche05}
	J.~Baik, G.~B.~Arous, and S.~P{\'e}ch{\'e}.
	\newblock Phase transition of the largest eigenvalue for nonnull complex sample
	covariance matrices.
	\newblock {\em Ann. Probab.}, 33(5):1643--1697, 2005.
	
	\bibitem{Baik-Ferrari-Peche10}
	J.~Baik, P.~L.~Ferrari, and S.~P{\'e}ch{\'e}.
	\newblock Limit process of stationary {TASEP} near the characteristic line.
	\newblock {\em Comm. Pure Appl. Math.}, 63(8):1017--1070, 2010.
	
	\bibitem{Baik-Ferrari-Peche14}
	J.~Baik, P.~L.~Ferrari, and S.~P{\'e}ch{\'e}.
	\newblock Convergence of the two-point function of the stationary {TASEP}.
	\newblock In {\em Singular phenomena and scaling in mathematical models}, pages
	91--110. Springer, Cham, 2014.
	
	\bibitem{Baik-Liu16}
	J.~Baik and Z.~Liu.
	\newblock Fluctuations of {TASEP} on a ring in relaxation time scale.
	\newblock 2016.
	\newblock arXiv:1605.07102.
	
	\bibitem{Baik-Liu16b}
	J.~{Baik} and Z.~{Liu}.
	\newblock Tasep on a ring in sub-relaxation time scale.
	\newblock 2016.
	\newblock {\em J. Statist. Phys.}, 165(6):1051--1085, 2016.
	
	\bibitem{Baik-Rains00}
	J.~Baik and E.~M.~Rains.
	\newblock Limiting distributions for a polynuclear growth model with external
	sources.
	\newblock {\em J. Statist. Phys.}, 100(3-4):523--541, 2000.
	
	\bibitem{Basu-Sidoravicius-Sly16}
	R.~Basu, V.~Sidoravicius, and A.~Sly.
	\newblock {Last Passage Percolation with a Defect Line and the Solution of the
		Slow Bond Problem}.
	\newblock arxiv:1408.3464.
	
	\bibitem{Derrida-Lebowitz98}
	B.~Derrida and J.~L.~Lebowitz.
	\newblock Exact large deviation function in the asymmetric exclusion process.
	\newblock {\em Phys. Rev. Lett.}, 80(2):209--213, 1998.
	
	\bibitem{Ferrari-Spohn06}
	P.~L.~Ferrari and H.~Spohn.
	\newblock Scaling limit for the space-time covariance of the stationary totally
	asymmetric simple exclusion process.
	\newblock {\em Comm. Math. Phys.}, 265(1):1--44, 2006.
	
	\bibitem{Gwa-Spohn92}
	L.-H.~Gwa and H.~Spohn.
	\newblock Six-vertex model, roughened surfaces, and an asymmetric spin
	{H}amiltonian.
	\newblock {\em Phys. Rev. Lett.}, 68(6):725--728, 1992.
	
	\bibitem{Johansson00}
	K.~Johansson.
	\newblock Shape fluctuations and random matrices.
	\newblock {\em Comm. Math. Phys.}, 209(2):437--476, 2000.
	
	\bibitem{Meakin-Ramanlal-Sander-Ball86}
	P.~Meakin, P.~Ramanlal, L.~M. Sander, and R.~C. Ball.
	\newblock Ballistic deposition on surfaces.
	\newblock {\em Phys. Rev. A}, 34:5091--5103, Dec 1986.
	
	\bibitem{Nejjar16}
	P.~Nejjar.
	\newblock Transition to shock fluctuations in {TASEP}.
	\newblock Upcoming.
	
	\bibitem{Prahofer-Spohn02a}
	M.~Pr{\"a}hofer and H.~Spohn.
	\newblock Current fluctuations for the totally asymmetric simple exclusion
	process.
	\newblock In {\em In and out of equilibrium (Mambucaba, 2000)}, volume~51 of
	{\em Progr. Probab.}, pages 185--204. Birkh\"auser Boston, Boston, MA, 2002.
	
	\bibitem{Prolhac16}
	S.~Prolhac.
	\newblock Finite-time fluctuations for the totally asymmetric exclusion
	process.
	\newblock {\em Phys. Rev. Lett.}, 116:090601, 2016.
	
\end{thebibliography}

\def\cydot{\leavevmode\raise.4ex\hbox{.}}

\end{document}